\RequirePackage{etex}
\documentclass[11pt,oneside]{amsart}
\usepackage{bbm}
\usepackage{mathrsfs}

\usepackage[english]{babel}
\usepackage[utf8]{inputenc}
\usepackage[T1]{fontenc}

\usepackage[top=3cm,bottom=2cm,left=3cm,right=3cm,marginparwidth=1.75cm]{geometry}
\usepackage{physics}
\usepackage{amsmath,amssymb}
\usepackage{esint}
\usepackage{graphicx}
\usepackage[dvipsnames]{xcolor}
\usepackage[colorlinks=true,allcolors=BlueViolet, pagebackref=true]{hyperref}
\renewcommand*\backref[1]{\ifx#1\relax \else (Cited on p.#1) \fi}
\hypersetup{hypertexnames=false}
\usepackage{cleveref}
\usepackage[shortlabels]{enumitem}
\usepackage[all,cmtip]{xy}
\usepackage{relsize}
\usepackage{subcaption}
\usepackage{dsfont}

\usepackage{xspace} 
\usepackage{cancel}
\usepackage{fancyvrb}

\DeclareFontFamily{U}{BOONDOX-calo}{\skewchar\font=45 }
\DeclareFontShape{U}{BOONDOX-calo}{m}{n}{
  <-> s*[1.05] BOONDOX-r-calo}{}
\DeclareFontShape{U}{BOONDOX-calo}{b}{n}{
  <-> s*[1.05] BOONDOX-b-calo}{}
\DeclareMathAlphabet{\mathbdx}{U}{BOONDOX-calo}{m}{n}
\SetMathAlphabet{\mathbdx}{bold}{U}{BOONDOX-calo}{b}{n}
\DeclareMathAlphabet{\mathbbdx}{U}{BOONDOX-calo}{b}{n}

\newtheorem{theorem}{Theorem}[section]
\newtheorem{proposition}[theorem]{Proposition}
\newtheorem{corollary}[theorem]{Corollary}
\newtheorem{lemma}[theorem]{Lemma}

\newtheorem{definition}[theorem]{Definition}

\theoremstyle{remark} 

\newtheorem{remark}[theorem]{Remark}

\crefname{equation}{Equation}{Equations}
\crefname{gather}{Equation}{Equations}
\crefname{multline}{Equation}{Equations}
\crefname{figure}{Figure}{Figures}
\crefname{question}{Question}{Question}
\crefname{section}{Section}{Sections}
\crefname{subsection}{Subsection}{Subsections}
\crefname{appendix}{Appendix}{Appendices}
\crefname{lemma}{Lemma}{Lemmas}
\crefname{proposition}{Proposition}{Propositions}
\crefname{theorem}{Theorem}{Theorems}
\crefname{corollary}{Corollary}{Corollaries}
\crefname{definition}{Definition}{Definitions}
\crefname{remark}{Remark}{Remarks}
\crefname{example}{Example}{Examples}
\crefname{claim}{Claim}{Claim}
\crefname{conjecture}{Conjecture}{Conjecture}
\crefname{yauconjecture}{Yau's conjecture}{Yau's conjecture}

\definecolor{bluola}{RGB}{138,43,226}

\newcommand{\R}{\mathbb{R}}

\newcommand{\N}{\mathbb{N}}
\newcommand{\PP}{\mathbb{P}}

\newcommand{\E}{\mathbb{E}}

\newcommand{\de}{\partial}

\newcommand{\inter}[1]{%
  {\kern0pt#1}^{\mathrm{o}}%
}
\renewcommand{\div}[1]{\mathrm{div}(#1)}

\newcommand{\f}{\varphi}
\renewcommand{\a}{\alpha}
\newcommand{\e}{\varepsilon}

\newcommand{\id}{\mathbbm{1}}
\newcommand{\vol}[1]{\mathrm{Vol}^{#1}}

\newcommand{\n}{n}
\renewcommand{\d}{\n}
\newcommand{\brkt}[2]{\left \langle #1,#2\right\rangle}

\newcommand{\Var}{{\mathbb{V}\mathrm{ar}}}
\newcommand{\Cov}{\mathrm{Cov}}
\newcommand{\coeff}{\Theta}
\newcommand{\lf}{\mathcal{L}_f}
\newcommand{\lfs}{\hat{\mathcal{L}}_f}
\newcommand{\hlf}{\mathcal{\Tilde{L}}_f}
\newcommand{\hlfi}{\mathcal{\Tilde{L}}_\phi}
\newcommand{\LL}{\mathrm{L}}

\newcommand{\berry}{\mathbdx{b}}

\def\randin{%
  \mathchoice%
    {\raisebox{-.35ex}{$\displaystyle{^\subset}$}\mkern-11.5mu\raisebox{+.45ex}{$\displaystyle{_\subset}$}}
    {\mkern+1mu\raisebox{-.27ex}{$\textstyle{^\subset}$}\mkern-11.7mu\raisebox{+.45ex}{$\textstyle{_\subset}$}}
    {\raisebox{.35ex}{$\scriptstyle\subset$}\mkern-14mu\raisebox{-.15ex}{$\scriptstyle\subset$}}
    {\raisebox{.3ex}{$\scriptscriptstyle\subset$}\mkern-13.5mu\raisebox{-.10ex}{$\scriptscriptstyle\subset$}}
}

\makeatletter
\newcommand{\tpitchfork}{%
  \raise-0.1ex\vbox{
    \baselineskip\z@skip
    \lineskip-.52ex
    \lineskiplimit\maxdimen
    \m@th
    \ialign{##\crcr\hidewidth\smash{$-$}\hidewidth\crcr$\pitchfork$\crcr}
  }%
}

\newcommand{\iid}{\emph{i.i.d.}\xspace}

\newcommand{\mC}{\mathcal{C}}
\newcommand{\m}[1]{\mathcal{#1}}
\newcommand{\be}{\begin{equation}}
\newcommand{\ee}{\end{equation}}
\numberwithin{equation}{section}

\newcommand{\bega}{\begin{equation}\begin{aligned}}
\newcommand{\eega}{\end{aligned}\end{equation}}

\newcommand{\begt}{\begin{equation}\begin{gathered}}
\newcommand{\eegt}{\end{gathered}\end{equation}}
\newcommand{\kop}{\left\{}
\newcommand{\pok}{\right\}}
\newcommand{\tyu}{\left(}
\newcommand{\uyt}{\right)}
\newcommand{\qwe}{\left[}
\newcommand{\ewq}{\right]}

\title[New chaos decomposition for Gaussian nodal volumes]{
New chaos decomposition of Gaussian nodal volumes}
\author{Michele Stecconi, Anna Paola Todino}
\date{\today}
\begin{document}

\begin{abstract}
%
%
We investigate the random variable defined by the volume of the zero set of a smooth Gaussian field, on a general Riemannian manifold possibly with boundary, a fundamental object in probability and geometry. We prove a new explicit formula for its Wiener-It\^o chaos decomposition that is notably simpler than existing alternatives and which holds in greater generality, without requiring the field to be compatible with the geometry of the manifold.
A key advantage of our formulation is a significant reduction in the complexity of computing the variance of the nodal volume. Unlike the standard Hermite expansion, which requires evaluating the expectation of products of $2+2n$ Hermite polynomials, our approach reduces this task—in any dimension $n$—to computing the expectation of a product of just four Hermite polynomials.
As a consequence, we establish a new exact formula for the variance, together with lower and upper bounds.


Importantly, in contrast to previous results, our approach applies to highly non-isotropic situations, allowing the study of Riemannian random waves on arbitrary manifolds. By introducing two parameters associated to any Gaussian field: the frequency and the eccentricity, we quantify the deviation from the standard settings (e.g., spheres) and establish a quantitative version of Berry’s cancellation phenomenon valid on all manifolds.


\end{abstract}
\maketitle
\tableofcontents 
\section{Introduction} 
\subsection{Notations}\label{sec:notations}

The following list contains some recurring conventions adopted in our work.
\begin{enumerate}[(i)]


\item A \emph{random element} (see \cite{Billingsley}) of the topological space $T$ (or \emph{with values} in $T$) is a measurable mapping $X\colon \Omega\to T$, defined on a probability space $\tyu \Omega,\mathscr{E},\PP \uyt$. In this case, we write
    \be\label{eq:randin}
    X\randin T
    \ee 
    and denote by $[X]=\PP X^{-1}$ the (push-forward) Borel probability measure on $T$ induced by $X$. We will use the notation
\be 
\PP\{X\in U\}:=
\PP X^{-1}(U)
\ee 
to indicate the probability that $X\in U$, for some Borel measurable subset $U\subset T$, and write (as usual)
\be 
\E\{f(X)\}:=\int_{T}f(t)d[X](t),
\ee
to denote the expectation of the random variable $f(X)$, where $f\colon T\to \R^k$ is a measurable mapping such that the above integral is well-defined.
We will say that $X$ is a \emph{random variable}, a \emph{random vector} or a \emph{random field}, respectively, when $T$ is the real line, a vector space, or a space of functions, 
respectively. 
\item The sentence: ``$X$ has the property $\mathcal{P}$ almost surely'' (abbreviated ``a.s.'') means that the set  $S=\{t\in T : t \text{ has the property }\mathcal{P}\}$ contains a Borel set of $[X]$-measure $1$. It follows, in particular, that the set $S$ is $[X]$-measurable, i.e. it belongs to the $\sigma$-algebra obtained from the completion of the measure space $(T,\mathcal{B}(T),[X])$.

\end{enumerate}

\subsection{Main result on chaotic decomposition}
\subsubsection{Gaussian nodal volumes}
Let $(M,g)$ be a $\d$-dimensional Riemannian manifold, possibly with boundary. For all $r\in \N$ or $r=\infty$, we denote by $\mC^r(M)$ the space of real valued functions on $M$ of class $\mC^r$. 
The \emph{nodal volume measure} of $f\in \mC(M)$ is the mapping 
\be \label{eq:vol1}
A\mapsto \lf(A)=\vol{\d-1}(f^{-1}(0)\cap A)=\int_A\lf(dx),
\ee 
defined for $A\subset M$ Borel subsets. 
That is, the restriction of the $(d-1)$-Hausdorff measure of $(M,g)$ to the \emph{nodal set} $f^{-1}(0)$ of $f$. The \emph{nodal volume} of $f$ is the real number $\lf(M)$. 
For the moment being, let us disregard the regularity of $M,g,f$ that we consider, which will be specified later.
Assuming that the space of continuous functions $\mC(M)$ is endowed with a Gaussian probability measure $\PP$ (see \cite{bogachev}) we study the random variable $f\mapsto \lf(M)$.
Equivalently, we consider \eqref{eq:vol1} under the hypothesis that $f$ is a continuous Gaussian random field on $M$, in the sense of \cite{AT07,AzaisWscheborbook}, i.e., a random function that admits a representation as
\be \label{eq:karlo}
f=\sum_{i=0}^{N}\gamma_i\f_i,
\ee
for some collection of functions $\f_i\in \mC(M)$ and a family of independent and identically distributed $\gamma_i\sim \mathcal{N}(0,1)$ (Karhunen-Lo\`{e}ve expansion \cite{bogachev}). Then, $f\randin \mC(M)$  defines a probability space $\tyu \mC(M), \mathscr{B},\PP \uyt$, with $\mathscr{B}$ being the Borel $\sigma$-algebra. 
Here, $N \in \N$, or $N=+\infty$; in the former case, the probability measure is defined by 
\be \label{eq:Gausprob}
\E\kop\alpha\tyu f \uyt\pok:=\int_{\mC(M)}\alpha\tyu f\uyt\PP(df):=
\int_{\R^N}\alpha\tyu\sum_{i=0}^{N}t_i\f_i\uyt\frac{e^{-\frac{\|t\|^2}{2}}}{(2\pi)^{\frac N2}} dt,
\ee 
for any bounded continuous functional $\alpha \colon \mC(M)\to \R$;
while if $N=+\infty$, the series \eqref{eq:karlo} 
is almost surely convergent in $\mC(M)$ and Equation \eqref{eq:Gausprob} should be corrected by adding the limit $\lim_{N\to +\infty}$ in front of the right hand side. If the convergence holds in $\mC^r(M)$, for some $r\in \N$ or $r=\infty$, then we write $f\randin \mC^r(M)$, see \cref{sec:notations}.

\begin{remark} 
Let us immediately point out that Gaussianity is irrelevant.
The functional $f\mapsto \alpha(f):=\lf(M)$ is 0-homogenous: $\alpha(tf)=\alpha(f)$ for all $t\in \R\smallsetminus \{0\}$, hence it should be regarded as a functional on the quotient space, namely on the projective space $\mathsf{P}(\R^N)$. 
All properties of the law of the random variable $\alpha(f)$, e.g. distribution, expectation, variance, 
remain valid for any choice of random coefficients $\gamma=(\gamma_1,\dots, \gamma_N)$ provided that ${\gamma}/{|\gamma|}$ induces the uniform measure on the sphere $S^{N-1}$, independently of the radial direction.
Taking Gaussian variables is however the most convenient choice, largely due to the stability of Gaussian measures under limits and linear mappings. Furthermore,
Gaussian analysis provides an efficient comprehensive framework to make sense of the infinite dimensional case $N=+\infty$ and gives access to powerful tools such as the Wiener-Itô chaos decomposition \cite{Nourdin_Peccati_2012}. 
\end{remark}
\subsubsection{Notation for measures}\label{subsec:notameas}
Given a measure $\mu$ on a measurable space $\Omega$, we denote the measure of a measurable subset $A$ and, when defined, the integral of a function $\alpha\colon \Omega\to \R$ as 
\be 
\mu(A)=\int_A\mu(dx) \quad \text{and} \quad \langle \mu, \alpha\rangle=\int_\Omega\alpha(x)\mu (dx),
\ee
respectively. In the following, we will formally identify the measure with its density $\mu(dx)$, interpreted in the distributional sense. Moreover, when $\mu=\vol{n}$ is the $n$-Hausdorff measure of a Riemannian manifold $(\Omega=M,g)$ of dimension $\dim(M)=\n$, we will just write
\be 
dx\equiv \vol{\n}(dx).
\ee
Such notation will keep the formulas more readable, while causing only minimal ambiguities; we already used it in Equation \eqref{eq:vol1} and in Equation \eqref{eq:Gausprob} above. Finally, if $\mu(A)\neq 0$, then we write 
\be 
\fint_A \a(x)\mu(dx):=\frac{1}{\mu(A)}\int_A\a(x)\mu(dx).
\ee
\subsubsection{Assumptions on the field}\label{subsec:assjet} Consider a $\mC^2$ Gaussian random field $f$ defined on a $\mC^2$ compact manifold $M$ of dimension $\n \in \N$, possibly with boundary, and endowed with a $\mC^1$ Riemannian metric $g$. 
Our only assumption on $f$, other than its $\mC^2$ regularity, is that it should have unit variance and non-degenerate differential $d_xf\in T_x^*M$, that is, we assume that 
\be\label{eq:assjet} 
\E\kop |f(x)|^2\pok=1,\quad \text{ and }\quad 
\|u\|_{g^f}^2:=\E\kop |d_xf(u)|^2\pok>0, 
\ee
for every $x\in M$ and $u\in T_xM\!\smallsetminus\! \{0\}$. Here, $\|\cdot\|_{g^f}$ can be recognized as the norm on $T_xM$ associated to the Adler-Taylor metric $g^{f}$, see \cref{eq:ATmetric} below. 

Given any $\mC^2$ Gaussian field $\phi$ on $M$, one can reduce to the above situation, via the normalization: 
\be\label{eq:unit_var_norm} 
f(\cdot)=\E\qwe\phi(\cdot)^2\ewq^{-\frac12}\phi(\cdot) \quad \implies \quad \m L_{\phi}=\lf,
\ee
provided that for all $x\in M$, the Gaussian vector $(\phi(x),d_x\phi)$ is non-degenerate in $\R\times T_xM$. Indeed, such condition holds if and only if $f$ satisfies Condition \eqref{eq:assjet}.
 In this case, we say that \emph{$\phi$ has non-degenerate first jet} and that $f$ is its \emph{unit-variance normalization}. Fields $\phi$ with this property are the same as those considered in \cite{nvdfg2024PeccatiStecconi}.

We will denote the covariance function of $f$ (and not that of $\phi$) as 
\be 
C(x,y):=\E\kop f(x)f(y)\pok,
\ee
for all $x,y\in M$.
The assumptions at \eqref{eq:assjet}
ensure that $C$ is of class $\mC^2$ in both variables, and that it has a Taylor expansion of the form:
$C(x,x+u)=1-\frac12\|u\|_{g^f}^2(1+o(1))$, at $u=0$.
\subsubsection{Chaos decomposition}
In \cite[Thm 1.5]{fom2024GassStecconi} (see also \cite{letendre2023multijet,AngstPoly}) it is shown that with such hypotheses, $\E\kop \lf(M)^2\pok <+\infty$. It follows that for any $A\subset M$ Borel, $\lf(A)$ can be written as an $L^2$-convergent series, called the \emph{Wiener-Ito chaos decomposition} (see \cite[Def. 2.2.3]{Nourdin_Peccati_2012} and \cref{sec:chaos})
\be 
\lf(A)=\sum_{q\in \N} \lf(A)[q],
\ee
where $\lf(A)[q]$ denotes the \emph{$q^{th}$ chaos component}, that is, the $L^2$-projection of $\lf(A)$ onto the closed subspace of $L^2$ generated by the random variables of the form $H_q(\gamma)$, where $H_q$ is the Hermite polynomial of degree $q$ and $\gamma \sim \m N(0,1)$ varies among all normal Gaussian variables, see \cref{sec:chaos}.
By linearity, the association $A\mapsto \lf(A)[q]$ is a random measure, which we denote as $\lf[q]$, for every $q\in 2\N$, and call it the \emph{$q^{th}$ chaos component of the nodal measure $\lf$}, see also \cite[Sec. 6.3]{letendre_varvolII_2019}. We recall that a key feature of such construction is that the random variables $\{\lf(A)[q]\}_{q\in 2\N}$ so obtained are uncorrelated, so that
\be 
\Var\kop \lf(A) \pok=\sum_{q\in \N,\ q\ge 1} \E\kop\tyu \lf(A)[q]\uyt^2\pok.
\ee 
In particular, each term $\lf(A)[q]$ provides a lower bound of the variance. 
\subsubsection{Main formula}
Given a Riemannian manifold, let $S(T_xM):=\kop v\in T_xM\colon g_x(v,v)=1\pok$ denote the unit tangent sphere at $x\in M$. Clearly, $S(T_xM)\subset T_xM$ is a Riemannian manifold with the metric induced by $g_x$, and it is isometric to the standard round sphere $S^{\n-1}$. 
\begin{theorem}\label{thm:nodalchaos}
Let $(M,g)$ be a compact Riemannian manifold of dimension $\d$, possibly with boundary, and $f\colon M\to \R$ be a Gaussian random field of class $\m C^2$ satisfying Condition \eqref{eq:assjet}, at every $x\in M$. Let $\|u\|_{g^f}
$ be the norm induced by $f$ on $T_xM$, as in \eqref{eq:assjet}. Then, the $q^{th}$ chaos component of the nodal measure of $f$ is the measure
\be\label{eq:nodalchaosAT} 
\begin{gathered}
\lf(dx)[q]
=
\\
\sum_{a,b\in \N,a+b=\!\frac{q}2}\!\frac{\coeff(a,b)}{s_\n} \int_{S(T_xM)} 
 H_{2a}(f(x))
 H_{2b}\tyu \frac{\langle d_xf, u\rangle}{\|u\|_{g^f}}\uyt \|u\|_{g^f}du dx,
\end{gathered}
\ee
for all $q$ even, while it is zero for all odd $q$. The constant $\coeff(a,b)$ is given by
\be\label{eq:defC} 
\coeff(a,b)=\frac{(-1)^{a+b-1}}{2^{a+b}(2b-1)a!b!},
\ee
for all $a,b\in \N$; while $s_n:=\vol n(S^n)=\frac{2 \pi^{\frac {n+1}2}}{\Gamma(\frac {n+1}2)}$ denotes the volume of the $n$-sphere.
\end{theorem}
The proof of \cref{thm:nodalchaos} will be given in \cref{sec:proof_th1.2}. As an immediate corollary of this result, by looking only at $q\in \{0,1,2\}$, we deduce the following variance lower bound.
\begin{corollary}\label{cor:secondchaosgeneral}
In the same setting as in \cref{thm:nodalchaos}, we have that 
\be\label{eq:1chaosgeneral} 
\begin{gathered}
\E\kop \lf(M) \pok
=
\frac{1}{ s_\n} \int_M\int_{S(T_xM)} \|u\|_{g^f}du dx,
\end{gathered}
\ee
and 
\be\label{eq:2chaosgeneral} 
\begin{gathered}
\lf(M)[2]
=
\frac{1}{2 s_\n} \int_M\int_{S(T_xM)} 
 \tyu -f(x)^2+
 \frac{\langle d_xf, u\rangle^2}{\|u\|_{g^f}^2} \uyt \|u\|_{g^f}du dx.
\end{gathered}
\ee
Moreover, $\Var\kop \lf(M)\pok \ge \E \kop (\lf(M)[2])^2\pok$. 
\end{corollary}
The variance of \cref{eq:2chaosgeneral} is very easy to compute explicitely in terms of the covariance function $C$, by means of the covariance properties of univariate Hermite polynomials (see \cite[p.13]{Nourdin_Peccati_2012}). We report such formulas in the homothetic case, that is, when $\|u\|_{g^f}$ is constant on $S(T_xM)$, in  \cref{appendix:homothetic}, together with that of the fourth chaos component.

\subsubsection{The two metrics}
The norm $\|\cdot\|_{g^f}$ defined in \cref{eq:assjet} is the norm on $T_xM$ associated to the Adler-Taylor tensor: 
$
x\mapsto g^f_x=\E\kop d_xf^{\otimes 2}\pok\in T^*_xM\otimes T^*_xM,$ that is,
\begt\label{eq:ATmetric}
g^f_x(u,v)=\E\kop \langle d_xf,u\rangle\langle d_xf,v\rangle \pok, \quad \forall x\in M,\ \forall u,v\in T_xM,
\eegt
as in \cite[Eq. (12.2.1)]{AT07} (see also \cite{elk2024PistolatoStecconi,ersz2024MathisStecconi}). Condition \eqref{eq:assjet} ensures that the tensor $g^{f}$ 
is a $\mC^1$ Riemannian metric on $M$, corresponding to \cite[Eq. (12.2.1)]{AT07}, and allows to translate the structural stochastic properties of $f$ in terms of the geometry of $(M,g^f)$. 
\begin{remark}
It is important to note that, at such level of generality, the two metrics $g^f$ and $g$ might have nothing to do one with the other, indeed every Riemannian metric $\Tilde{g}$ can be realized as $\Tilde{g}=g^f$ for a suitable Gaussian field $f$ satisfying the initial assumptions; this is a consequence of Nash' isometric embedding theorem, see \cite{AT07,elk2024PistolatoStecconi,ersz2024MathisStecconi}. 
\end{remark} 

The presence of two different metrics is a major source of complications, due to the fact that, on the left hand side of \cref{eq:nodalchaosAT}, we are interested in the nodal volume $\lf(M)$ with respect to the original metric $(M,g)$; while at the same time, all the quantities related to $f$ are much easier to manipulate if written in terms of its relative metric $g^f$. This explains the role of the term $\|u\|_{g^f}$ in the right hand side of \cref{eq:nodalchaosAT}. For instance, if $g^f=g$, then \cref{cor:secondchaosgeneral} yields that the second chaos cancels whenever $f$ satisfies an eigenvalue equation $\Delta f=\sqrt{n}f$ with standard boundary conditions, see \cref{thm:homo24} below.
\subsubsection{Homothetic fields} \label{sec:homotheticfields}
In most cases considered in the literature, one has that $g^f=\xi^2 g$, for some constant $\xi^2$, so that $\|u\|_{g^f}=\xi \|u\|$.
In \cite{elk2024PistolatoStecconi}, this situation is discussed in more details, and fields with such property are named \emph{homothetic fields}. We take up the same terminology in this paper.
In particular, homothetic Gaussian fields include: Gaussian isotropic and stationary fields on $\R^\d$, such as Berry's, considered for instance in \cite{Berry2002,DalmaoEstradeLeon,NPV23,NourdinPeccatiRossi2017,PeccatiVidotto2020,MAINI2024,maini2025}, or Bargmann Fock models, see \cite{2017_Beffara_Gayet, rivera:hal-03320870,DuminilVanneauville_2023}; Gaussian hyperspherical harmonics or Kostlan polynomials on $S^\d$ \cite{EdelmanKostlan95,kostlan:93}, studied for example in \cite{marinucci2023laguerre,Nazarov2007, Wig09, Wig10,
MRW20} and \cite{ RootsKost_ancona_letendre, stec2019MaxTyp,2017_Beffara_Gayet,LerarioFLL} respectively; 
arithmetic random waves on the standard flat torus $(S^1)^\d$, investigated in \cite{RudWig08,WigmanAnnMath,BMW,MPRW,Cammarota2017NodalAD,BM19,ORW08} among others. In most cases, this is an easy consequence of the invariance in law under a large enough group of symmetries. 
However, this is not the case for a general Riemannian Random Wave $\phi_\lambda$ on an unspecified Riemannian manifold $(M,g)$, of the kind introduced by Zelditch in \cite{Zel09} and studied for instance in \cite{CH20,Keeler,wigEBNRF}. 
\begin{remark}
The fact that the \cref{thm:nodalchaos} includes the case of non-homothetic fields (after renormalizing as in \cref{eq:unit_var_norm}) is a key novelty of this paper. 
\end{remark}
In the following, we will analyze the formula of \cref{thm:nodalchaos}, with the idea that the field is approximately homothetic. To this end, we write informally
\be\label{eq:ansax} 
g^{f}=\frac{\lambda^2}{\n} \tyu g+O(\e) \uyt,
\ee
meaning that $O(\e)$ is a tensor that is uniformly bounded by $\e$, 
with $\e=\e(f)\ge 0$ and $\lambda=\lambda(f)>0$ being deterministic positive parameters attached to $f$, as in \cref{def:three} below. We anticipate that the parameter $\e$ quantifies the extent to which the field $f$ is not homothetic; in other words, it measures the distortion of the spheres of $(T_xM,g^f_x)$ with respect to those of $(T_xM,g_x)$, therefore we will call it the \emph{eccentricity} (see \cref{def:three}). While we have in mind that $\e\to 0$ (for reasons explained in \cref{sec:MRW}), all our results are of a \emph{static} nature, i.e., quantitative.

The parameter $\lambda$ will be called \emph{frequency} (see \cref{def:three}) because if the field $f$ satisfies the equation $\Delta f=e f$, for some number $e\in \R$, then
\be 
\lambda^2=\tr_{g}\tyu g^f_x\uyt=-\E\kop f(x)\cdot \Delta f(x)\pok=-e.
\ee
This can be easily deduced by differentiating twice the identity $\E\kop f(x)^2\pok=1$ and then taking the trace in an orthonormal basis of $(T_xM,g_x)$. 

An additional source of complication is the step from a general field $\phi$ and its unit variance normalization $f$, defined as in \cref{subsec:assjet}. We will consider this issue in \cref{sec:introredhom}, by introducing a third parameter $\sigma=\sigma(\phi)> 0$, expressing the average variance, so that $\sigma(f)=1$ (see \cref{def:three}).

\subsubsection{Comparison with the usual chaotic expansion}
Using Federer's co area formula \cite{federer2014} one may formally express the nodal volume measure as
\begin{equation}\label{coarea}
\m L_f(dx)= \delta_0 (f(x)) \cdot \| \nabla f(x)\|dx,
\end{equation}
where $\delta_0$ is the Dirac delta function and $\|\cdot \|$ is the norm induced by the original metric $g$. Assumption \eqref{eq:assjet} implies that $f(x)$ and $\nabla f(x) $ are independent for fixed $x$, so that one can obtain the chaotic expansion of $\lf(dx)$ by multiplying those of $\delta_0(f(x))$ and of $\|\nabla f(x)\|$; this is a standard technique, see for instance \cite{kratzleon, MRW20, MPRW,Cammarota2017NodalAD}.
When $g=g^f$, the formula in \cref{thm:nodalchaos} is obtained directly from a new expansion of the chi random variable $\|\xi\|$ defined by a standard Gaussian vector $\xi$ in $\R^\n$:
\be \label{eq:chichaos}
\|\xi\|= \sum_{q\in \N} A(\d,q)\int_{S^{\d-1}}H_q\tyu \langle \xi, v\rangle\uyt dv;
\ee
where $A(\n,q)$ is given in \cref{A}. For general $g^f$, one needs a slightly more complicated formula, see \cref{thm:genchicaos}.
That such a formula exists is rather obvious from the spherical symmetry of the random variable $\|\xi\|$. 
However, such representation has never been considered in the literature so far (except it appears for instance, in \cite{Not23, LETENDRE2016}), due to the fact that it is not written in terms of the natural orthogonal basis formed by Hermite monomials: 
\be\label{eq:multihermite} 
H_\a(\xi):=H_{\a_1}\tyu\langle \xi, e_1\rangle \uyt\cdot \dots \cdot H_{\a_\d}\tyu\langle \xi, e_\d\rangle \uyt.
\ee 
Following this path leads to the, by now standard, formula for the chaotic expansion \cite[Lemma 4.1]{Cammarota2017NodalAD} (see also \cite[Section 6.3]{letendre_varvolII_2019}), that is,
\begin{eqnarray} \label{eq:chaosexpVALE}
\mathcal{L}_{f}(M)[2q] &= & 
        \sum_{p=0}^{q} \frac{\beta_{2q-2p}}{(2q-2p)!}\nonumber \\
		&&\times \!\!\!\! \sum_{\substack{s=(s_1,\dots, s_\d) \in \mathbb{N}^\d  \\ s_1+\dots+s_\d=p}} \frac{\alpha_{2s_1,\dots,2s_\d}}{(2s_1)!\dots(2s_\d)!} \int_{M} H_{2q-2p}(f(x)) \prod_{j=1}^{\d} H_{2s_j}({\partial}_j f(x))\,dx,
	\end{eqnarray}
	for even $q\geq 2$,
	where 
	$\beta_{2q-2p}= \frac{1}{\sqrt{2\pi} } H_{2q-2p}(0)$ and, for $s=(s_1,\dots, s_\d)\in \mathbb{N}^\d$,
\begin{eqnarray}
\alpha_{2s_1,\dots,2s_\d}= &&\sum_{i=0}^{\infty} \frac{1}{i!2^i} \frac{\sqrt{2}\Gamma(\frac{\d}{2}+i+\frac{1}{2})}{\Gamma( \frac{\d}{2}+i)} \nonumber \\
	&\times& \sum_{\substack{j_1+\dots+j_\d=i\\ j_1\le s_1, \dots, j_\d\le s_\d}} \binom{i}{j_1,\dots,j_\d} \frac{(-1)^{s_1-j_1+\dots+s_\d-j_\d}(2s_1)!\dots(2s_\d)!}{(s_1-j_1)!\dots(s_\d-j_\d)! 2^{ s_1-j_1+\dots+s_\d-j_\d}}.\nonumber \\
	\label{adef}
	\end{eqnarray}
One can then compress the latter, in terms of the generalized Laguerre polynomials $L^{_{(\alpha)}}_k$, as in \cite[Theorem 3.4]{marinucci2023laguerre}), so to obtain \cite[Remark 3.6]{marinucci2023laguerre}. Such a representation, is equivalent to that given in \cref{thm:nodalchaos}, indeed we will show that 
\be \label{eq:lag}
L_{\frac q2}^{(\frac \d2 -1)}\tyu \frac{\|\xi\|^2}{2}\uyt=c_{\d,q} \int_{S^{\d-1}}H_q\tyu \langle \xi, v\rangle\uyt \mathcal{H}^{\d-1}(dv),
\ee 
for some explicit constant $c_{\n,q}$,    see 
\cref{sec:laguerre}. However, we stress the fact that we will not need to manipulate the Laguerre polynomials in this paper; they will be mentioned only to compare our results with those of \cite{marinucci2023laguerre} and for general context.
\subsubsection{Exact formula for the Variance}\label{sec:exactvar}
The standard representation in terms of the Hermite basis \ref{eq:multihermite} requires the full power of the diagram formula (see \cite[Lemma 5.2]{CARAMELLINO2024110239} and \cite[Section 4.3.1]{marinucci_peccati_2011}) for computing expectations of the form 
\be 
\E\kop H_a(f(x))\prod_{i=1}^{\d} H_{s_i}(\partial_i f(x))H_{a'}(f(y))\prod_{j=1}^{\d} H_{s_j}(\partial_j f(y))\pok,
\ee 
involving the product of $O(2+2\d)$ Hermite polynomials. 

Besides the fact that \cref{eq:chichaos} can be proven directly (see \cref{thm:genchicaos}) and that it reflects the natural symmetry of the functional, such formula has the technical advantage of involving only one univariate Hermite polynomial, instead than the multivariate version written in \cref{eq:multihermite}. This simple fact simplifies significantly the computations of the variance of $\lf(M)[q]$, in any dimension $\d,$ which reduces to a formula for terms of the form:
\be \label{eq:seelemma}
\E\kop H_{2a}\tyu f(x)\uyt H_{2b}\tyu \langle d_xf, u\rangle\uyt
H_{2a'}\tyu f(y)\uyt H_{2b'}\tyu \langle d_yf, v\rangle\uyt
\pok,
\ee 
described in
\text{\cref{lem:diagfour}},
thus allowing us to write a treatable exact formula for the variance, which we report in \cref{cor:exactvar} below.
The representation of Laguerre polynomials at \cref{eq:lag}, combined with \cref{lem:diagfour}, complements the findings of \cite{marinucci2023laguerre}, by giving a formula to express the covariance of two Laguerre polynomials. 
\subsection{Covariance bound}
Our second main result is a general bound on the variance of each single chaotic component $\lf(M)[q]$, derived from the exact formula for the variance in \cref{cor:exactvar}. We are not aiming at a bound that is uniform in $q$, here. The variance is naturally expressed as an integral over $M\times M$: 
\be 
\Var\kop \lf(M)[q]\pok=\int_{M\times M} \E\kop \lf(dx)[q]\cdot \lf(dy)[q]\pok.
\ee 
A convenient natural way to study the variance is to bound directly the integrand in the above expression. We do this in \cref{cor:var_dx}, in terms of the derivatives of the covariance function $C$, encoded in a function $\|j''_{x,y}C\|_{g^f}$ of $(x,y)\in M\times M$, defined as in \cref{def:normcov1jet} below. 
\subsubsection{Covariance of the first jet}\label{sec:cov1jet}
Following the language of differential geometry, let us define the \emph{first jet} of $f$ at $x$ as the random vector
\be 
j^1_xf:=\tyu f(x),d_xf\uyt \randin\ \R\times T_x^*M,
\ee
where we recall that the elements of the cotangent space $T_x^*M$, called covectors, are the linear functions $T_xM\to \R$.
The covariance matrix of $j^1_xf$ and $j^1_yf$ is the collection of: the real number $C(x,y):=\E\kop f(x)f(y)\pok$; the two covectors $C'_{x,y}(\cdot):=\E\kop f(x)d_yf(\cdot)\pok\in T_y^*M$ and $C'_{y,x}(\cdot)\in T_x^*M$; and the bilinear function 
$
C''_{x,y}(\cdot,\cdot):=\E \kop 
d_xf(\cdot)d_yf(\cdot)\pok$ on $T_xM\times T_yM$, which thus defines an element of $ T_x^*M\otimes T_y^*M$.
In particular, $C''_{x,x}$ is the metric associated to $f$ as in \cref{eq:ATmetric}: $C''_{x,x}=g^{_f}_x$.
\begin{definition}
We gather the three objects together in one tensor, defined as
\be 
j''_{x,y}C:=\Cov\tyu j^1_xf,j^1_yf\uyt:=\begin{pmatrix}
C({x,y}) & C'_{x,y} \\ C'_{y,x} & C''_{x,y}
\end{pmatrix}\in (\R\times T_x^*M){\otimes }(\R\times T_y^*M)=:J^{1,1}_{x,y}(M),
\ee
and call it the \emph{covariance of the first jet}. 
\end{definition}
\begin{remark}
Plainly, $j''_{x,y}C$ is a section of the vector bundle $J^{1,1}(M)\to M\times M$ with fiber $J^{1,1}_{x,y}(M)$. 
Passing to a coordinate chart, one reduces to the case $M=T_xM=T_x^*M=\R^\n$, for all $x\in \R^\n$. Thus, $j^1_xf\randin \R^{\n+1}$ and $J^{1,1}_{x,y}(\R^\n)=\R^{(\n+1)\times (\n+1)}$. In such case, we have that $C'_{x,y}=\frac{\de C}{\de y}(x,y)\in \R^\n$ and $ C''_{x,y}=\frac{\de^2 C}{\de x\de y}(x,y)\in \R^{\n\times \n}$.
Delving deeper in the language of jets (see \cite{hirsch}), $j''C$ could be naturally interpreted as a jet of the covariance function $C$. 
\end{remark}

Since $\lf(dx)[q]$ is an integral of functions of $j^1_xf$, it follows that $\E\{\lf(dx)[q]\lf(dy)[q]\}$ is an integral of functions of $j''_{x,y}C$, as it can be directly seen from \cref{cor:exactvar} below.
\begin{definition} \label{def:normcov1jet}
We define a norm on the space $J^{1,1}_{x,y}(M)$, induced in a natural way from the norm $\|\cdot \|_{g^f}$ defined in \cref{eq:assjet}.
\be\label{eq:jtwonorm}
\|j''_{x,y}C\|_{g^f}:=\max_{
\substack{
u\in T_xM\smallsetminus\{0\} \\ v\in T_yM\smallsetminus\{0\} 
}
}\kop 
C(x,y),
\frac{|C'_{y,x}(u)|}{\|u\|_{g^f}},
\frac{|C'_{x,y}(v)|}{\|v\|_{g^f}},
\frac{|C''_{x,y}(u,v)|}{\|u\|_{g^f}\|v\|_{g^f}}
\pok.
\ee
\end{definition}
This will serve as a measure of the covariance, adapted to the field $f$ itself. For this reason, it is easy to see that $\|j''_{x,y}C\|_{g^f}\le 1$ by Cauchy-Schwarz. 
\begin{definition}\label{def:pointwisefreq}
To compare the norm $\|u\|_{g^f}$ with the original one $\|u\|_g=\sqrt{g(u,u)}$, we introduce the \emph{pointwise frequency of $f$} as the function $x\mapsto \lambda(f,x)>0$ such that
\be 
\lambda(f,x)^2:=
\E\kop \|d_xf\|^2 \pok=\n \fint_{S(T_xM)}\|u\|_{g^f}^2 du,
\ee
for all $x\in M$.
\end{definition}
\subsubsection{Main result 2: covariance inequality}
\begin{theorem}\label{cor:var_dx}
Let $f:M\to \mathbb{R}$ be as in \cref{thm:nodalchaos}. Then
\bega
|\E\kop \lf(dx)[q]\cdot \lf(dy)[q]\pok|
\le \
2^q\cdot\frac{\lambda(f,x)\lambda(f,y)}{\n}
\cdot
{\|j''_{x,y}C\|_{g^f}^q}
\ dxdy,
\eega
so that
\bega
\Var \tyu \lf(M)[q]\uyt 
\le \
{2^q}
\cdot
\int_{M\times M} \frac{\lambda(f,x)\lambda(f,y)}{\n} \cdot {\|j''_{x,y}C\|_{g^f}^q}dxdy.
\eega
\end{theorem}
This theorem is proven in \cref{Sec:cov_bounds}.


\subsection{Reduction to the homothetic case}\label{sec:introredhom}
\subsubsection{The average frequency and the maximal eccentricity}
Consider a Gaussian field $\phi$ on $(M,g)$ with non-degenerate first jet, and let $f$ be its unit-variance normalization, as in \cref{subsec:assjet}.
\begin{definition}\label{def:three}
We define three deterministic positive real numbers associated to $M,g,\phi$: the \emph{average variance $\sigma$}, the \emph{average frequency $\lambda$}:
\be\label{eq:sigmalambda} 
\sigma^2=\sigma(\phi)^2:=\fint_M \E\kop \phi(x)^2\pok dx,\quad \lambda^2=\lambda(\phi)^2:=\frac{1}{\sigma^2}\fint_M \E\kop \|d_x\phi\|^2\pok dx,
\ee 
and the \emph{maximal eccentricity} of $\phi$:
\bega 
\e=\e(\phi):=\max_{x\in M,\ u\in S(T_xM)}\left|\frac{\sqrt{\E\kop |\langle d_x\phi,u\rangle|^2\pok}}{\sigma\lambda}\sqrt{\n}-1\right| 
\\
+\max_{x\in M}\left|\frac{\sqrt{\E\kop \phi(x)^2\pok}}{\sigma}-1\right|
+
\max_{x\in M,\ u\in S(T_xM)}\left|\langle d_x\frac{\sqrt{\E\kop \phi(\cdot)^2\pok}}{\sigma}, u\rangle \right|\frac{\sqrt{\n}}{\lambda}.
\eega
When $\e(\phi)=0$, we say that $\phi$ is a \emph{homothetic field}.
\end{definition}
We will use the above parameters to quantify the error produced by writing the formula \eqref{eq:nodalchaosAT}, directly with $\phi$ and without the term $\|v\|_{g^\phi}$, as in \eqref{eq:defhom} below. Observe that, $\sigma(f)=1$, $\lambda(f)=\lambda(\phi)$ and $\e(f)\le \e(\phi)$, with equality if and only if $\phi=\sigma(\phi) \cdot f$.
\begin{definition}\label{def:esphom}
For all $q\in 2 \N$, we define the random variable 
\be\label{eq:defhom}
\hlfi(M)\{q\}:=\sum_{a,b\in \N,\ a+b=\frac{q}2}\frac{\coeff(a,b)}{s_\n\sqrt{\n}}\int_M\int_{S(T_xM)} 
 H_{2a}\tyu \frac{\phi(x)}{\sigma}\uyt
 H_{2b}\tyu \frac{\langle \lambda^{-1}\nabla_x\phi, v\rangle}{\sigma} {\sqrt{\n}}
 \uyt 
 \lambda dv dx,
\ee
where $\sigma=\sigma(\phi), \lambda=\lambda(\phi)$ are defined as in \eqref{eq:sigmalambda}.
\end{definition} In the homothetic case, when the eccentricity $\e$ is zero, we have that $\phi=\sigma\cdot f$ and that $\|u\|_{g^\phi}=\frac{\lambda}{\sqrt{\n}} \|u\|_g$, so that formula \eqref{eq:nodalchaosAT} reduces to $\lf(M)[q]=\hlfi(M)\{q\}$. In general, for $\e\neq 0$, the random variable at \cref{eq:defhom} might not even be in the $q^{th}$ chaos space; this is because the variable $\gamma=\langle \nabla_x\phi, v\rangle \frac{\sqrt{\n}}{\sigma\lambda}$ might not have variance one, in which case $H_{2b}(\gamma)$ is not an element of the $(2b)^{th}$ chaos. However, it can be easier to manipulate such expression and we show that it approximates the true chaotic component, up to an error that is linear in $\e$ in $L^2$, in the sense of the next theorem. 
\begin{theorem}\label{thm:non-homotetic}
 Let the setting of \cref{subsec:assjet} prevail. Then, there are constants $\kappa_{\ell,q}\in\R$, depending only on $(\ell,q)\in \N^2$, such that
\be 
\left\| \m L_\phi(M)[q]- \hlfi(M)\{q\} [\ell]\right\|_{L^2}\le  \e(\phi) \lambda(\phi)s_{\n-1} \kappa_{\ell,q}\sqrt{\int_{M\times M} \|j''_{x,y}C\|_{g^f}^\ell} .
\ee
for all $\ell\le q\in \N$, with $\kappa_{2\ell+1,q}=0$; $\kappa_{0,2}=0$.
\end{theorem}
We will prove \cref{thm:non-homotetic} in \cref{sec:redhom}. From standard theory of chaotic decomposition, one can deduce the lower bound on the variance of $\m L_{\phi}$.
\begin{corollary} In the setting of \cref{thm:non-homotetic} above, we have
\be 
\Var\tyu \frac{\m L_\phi(M)}{\lambda}\uyt\ge \Var\tyu \frac{\hlfi(M)\{2\}}{\lambda} \uyt+\Var\tyu \frac{\hlfi(M)\{4\}}{\lambda} \uyt+O_\n(\e)\int_{M\times M} \|j''_{x,y}C\|_{g^f}^2,
\ee
where $O_\n(\e)$ denotes a quantity that is bounded by a dimensional constant $c_\n$ times $\e=\e(\phi)$.
\end{corollary}
\begin{proof}
    A straightforward consequence of \cref{thm:non-homotetic} and \cref{cor:var_dx}.
\end{proof}
\subsubsection{Zeroth, second and fourth chaos}
We report the specialization of \cref{thm:nodalchaos} to the first chaotic components, up to the fourth. We are mostly interested in the situation of approximately homothetic field, that is, in the regime $\e=\e(\phi)\to 0$.

Since the order of the expectation
\be 
\frac{\E\kop \m L_\phi(M)\pok}{\lambda}=\frac{s_{\n-1}}{s_\n\sqrt{\n}} \vol{\n}(M)(1+O(\e))
\ee 
is the frequency $\lambda=\lambda(\phi)$, we divide all formulas by $\lambda$. For symmetry reasons, $\m L_\phi[2a+1]=0$ for all $a\in \N$, thus we restrict to even chaoses. We report the formula up to terms that are bounded by $\e$ in $L^2$. 
This is provided by \cref{thm:non-homotetic} above, yielding in particular that
\be\label{eq:approxchaos}
\frac{\m L_\phi(M)[q]}{\lambda}= \frac{\hlfi(M)\{q\}}{\lambda}+O_{L^2}\tyu \e \uyt
\ee
 where the expression $v=O_{L^2}(\e)$ is used to denote that there exists a constant $\kappa>0$, depending on $q$, on $\n$ and on the manifold $(M,g)$, such that the $\E|v|^2\le \kappa \e^2$. Moreover, we give the formulas in terms of the eigenvalue operator
$1+\lambda^{-2}{\Delta}, 
$ to show how $\lambda$ is related to Laplace eigenvalues, via an integration by parts formula. 

Let $\nu\colon \de M \to TM|_{\de M}$ be the outer unit normal vector to the boundary. For instance, if $M=\kop h\le u\pok $ is a regular excursion set for some smooth function $h\in \mC^\infty(M')$ defined on a possibly larger manifold $M'$, then $\nu(x)={\nabla h(x)}{\|\nabla h(x)\|^{-1}}$ for all $x\in \de M=h^{-1}(u)$. 
\begin{proposition}\label{thm:homo24}
Let the setting of \cref{subsec:assjet} prevail and let $\lambda=\lambda(\phi)$. The second and fourth chaos can be approximated up to an $L^2$ error of order $\e=\e(\phi)$ in the sense of \cref{thm:non-homotetic} by
\bega 
\frac{\hlfi(M)\{2\}}{\lambda} 
 = &-\frac{1}{\sigma^2}\frac{s_{\n-1}}{2s_\n\sqrt{\n}}  \tyu \|\phi\|^2_{L^2(M)}-
  \|\lambda^{-1}\nabla \phi\|^2_{L^2(M)}\uyt
  \\
  =&-\frac{1}{\sigma^2}\frac{s_{\n-1}}{2s_\n\sqrt{\n}} \cdot 
\tyu\int_{M} \phi \tyu \phi+\frac{\Delta \phi}{\lambda^2}\uyt
- \int_{\partial M} \phi \langle \lambda^{-1}\nabla \phi , \nu \rangle \uyt;
\eega
\bega
\frac{\hlfi(M)\{4\}}{\lambda}= &\frac{1}{\sigma^4}\frac{s_{\d-1}}{8 s_\n\sqrt{\n}}\int_M 
\phi^2(\phi^2-4\sigma^2)
- \frac{\n}{\n+2}\|\lambda^{-1}\nabla \phi\|^4
-2{\phi^2\|\lambda^{-1}\nabla \phi\|^2}
 +
4\sigma^2{\|\lambda^{-1}\nabla \phi\|^2}
\\
=&\frac{1}{\sigma^4}\frac{s_{\n-1}}{24 s_\n\sqrt{\n}} \Bigg[\int_M 
 \phi^4
-  \frac{3\n}{\n+2}\|\lambda^{-1}\nabla \phi\|^4
\cr
 &+2(\phi^3-6\sigma^2\phi) \tyu \phi+\frac{\Delta \phi}{\lambda^2}\uyt dx
 \ 
 -
\frac{1}{\lambda}\int_{\de M}2(\phi^3-6\sigma^2\phi) \langle \lambda^{-1}\nabla \phi,\nu\rangle \ 
\Bigg].
 \\
 =&
 \frac{s_{\n-1}}{24s_\n \sqrt{\n}} \bigg[ 
 \int_M 
 H_4\tyu\frac{\phi}{\sigma}\uyt - 
 \fint_{S(T_xM)} 
  H_4\tyu \frac{\langle \lambda^{-1}\nabla\phi, v\rangle}{\sigma}\sqrt{\n}\uyt dv  
 \\& \quad  +\frac{2}{\sigma}H_3\tyu\frac{\phi}{\sigma}\uyt\tyu \phi+\frac{\Delta \phi}{\lambda^2}\uyt
 dx \ -  \frac{2}{\lambda\sigma}\int_{\partial M} H_3\tyu\frac{\phi}{\sigma}\uyt \langle \lambda^{-1}\nabla \phi, \nu
 \rangle  dx \bigg].
\eega
\end{proposition}
We reported three different versions of the formula for $\lambda^{-1}\hlfi(M)\{4\}$, for the reader's convenience. The last one, perhaps the most insightful, is given in terms of the Hermite polynomials
\be 
H_3(t)=t^3-3t,\quad \text{and}\quad H_4(t)=t^4-6t^2+3,
\ee
and it is thus the most efficient for variance computations.

Note that the formulas simplify significantly when $\de M=\emptyset$ and for eigenfunctions: $\Delta \phi=-\lambda^2 \phi$. In such case, $\m L_f[2]=\lambda O_{L^2}(\e)$, and is zero whenever $\phi$ is homothetic. The same holds on a manifold with boundary, provided that $\phi$ also satisfies a standard boundary condition among $\phi|_{\de M}=0$, or $\langle \nabla \phi|_{\de M},\nu\rangle=0$.
The proof of \cref{thm:homo24} can be found in \cref{subsec:proof_prop_homothetic}.
We report in \cref{appendix:homothetic} an exact formula for the variance of the second and fourth components $\lf(M)[2]$ and $\lf(M)[4]$ valid in the homothetic case $\e=0$, see \cref{thm:homexactvar2} and \cref{thm:homexactvar4}.
\subsection{Level sets, cosine transform and varifolds}
Our method of proof of \cref{thm:nodalchaos} extends easily to give a chaos decomposition of a more general class of functionals. A first obvious example is $\int_{f^{-1}(0)}h(x)dx=\int_Mh(x) \lf(dx)$, where $h\in \mC^0(M)$, which follows directly from the measure-theoretic interpretation. A second generalization consists in considering non-zero level set $f^{-1}(t)$, or, equivalently, a non-centered Gaussian field $f(\cdot)-t$, for some $t\in \R$. 
Finally, evidently the formula of \cref{thm:nodalchaos} should imply the chaotic decomposition of the integral of a certain class of functions $h\colon S(TM)\to \R$ on the spherical tangent bundle, although to make such statement precise is not trivial. In \cref{cor:cosine} we identify such class and extend the formula of \cref{thm:nodalchaos} to 
\be \label{eq:cosine}
\tyu \int_{f^{-1}(t)}h\tyu x,\frac{d_xf}{\|d_xf\|}\uyt dx \uyt [q]
=
\int_{M}\int_{S(T_xM)} \hat{\m L}_{f-t}(dxdv) [q]\cdot \mu_x(dv) dx,
    \ee
    where $\hat{\m L}_{f-t}(dxdv) [q]$ is given in \cref{eq:luxabmean} below. The formula is 
    valid for all $t\in \R$ and for $h$ being the cosine transform (see \cite{bible,MathisZA} and \cref{sec:cosine}) of a continuous section $x\mapsto \mu_x$ of measures on $S(T_xM)$, see \cref{cor:cosine}. 
    
We point out that, \cref{eq:cosine} means that $\hat{\m L}_{f-t}(dxdv)[q]$ yields a partial---because it holds for a restricted class of continuous functions $h$ on $S(TM)$---chaos decomposition of the random \emph{varifold} defined by $f^{-1}(t)$, see \cite{ersz2024MathisStecconi} for a thorough discussion of such point of view.

\subsection*{Acknowledgments}
We are grateful to Giovanni Peccati and Domenico Marinucci for their time and feedback throughout the preparation of this paper. We also thank Hermine Biermé and Anne Estrade for suggesting the inclusion of the case of general level sets, and Léo Mathis for clarifications on the cosine transform.

Michele Stecconi acknowledges support from the Luxembourg National Research Fund (Grant: 021/16236290/HDSA). 
Anna Paola Todino would like
to thank support from the GNAMPA-INdAM project 2024 "Geometria di onde
aleatorie su varietà" and the University of Luxembourg  where the present project took shape.
\section{Motivations and examples}

\subsubsection{Motivating example: Riemannian Random Waves}\label{sec:MRW}
On a compact Riemannian manifold $M$ without boundary, the eigenvalues of the Laplace-Beltrami operator $\Delta$ are all positive and can be ordered as an increasing diverging sequence: $\lambda_0^2=0< \lambda_1^2\le \lambda_2^2\le \dots \le \lambda_i^2\to +\infty$, possibly with repetitions. The corresponding eigenfunctions can be selected so that they form a complete orthonormal system $(\f_i)_{i\in \N}$ in $L^2(M)$. 
Given an interval $I\subset \R$, we call a \emph{Riemannian Random Wave} the Gaussian field
\be \label{eq:RW}
\phi_I(x)=\sum_{\lambda_i\in I}\gamma_i \f_i(x),\qquad \Delta \f_i=-\lambda_i^2\f_i
\ee
with $\gamma_i\sim \m N(0,1)$ \iid Such fields and terminology were introduced by Zelditch in \cite{Zel09}. For simplicity, in the rest of this section, we will assume that $\vol{}(M)=1$.
It is easy to compute the parameters $\sigma,\lambda$ and $\e$ of \cref{def:three} for such fields, which in general have no reason for being homothetic.
\be\label{eq:phiI} 
\e(\phi_I)\ge 0;
\quad 
\sigma(\phi_I)^2=\dim \tyu \bigoplus_{\lambda_i\in I}\ker(\Delta+\lambda_i^2)\uyt;\quad \lambda(\phi_I)^2=\frac1{\sigma^2}\sum_{\lambda_i\in I}\lambda_i^2.
\ee
From this is clear that $\lambda(\phi_I)^2$ is the average eigenvalue in $I$, thus our choice of calling $\lambda$ the \emph{average frequency}, in \cref{def:three}. Indeed, we can view $\lambda^2=\lambda(\phi_I)^2$ as the mean value $-\lambda^2=\E(\mu_I)$ of the \emph{spectral} probability measure $\mu_I:=\frac1{\sigma^2}\sum_{\lambda_i\in I}\delta_{-\lambda_i^2}$ on $\R_+$. Note that $\mu_I$ is the measure supported on the set of eigenvalues contained in $-I$, weighted by their multiplicity. The variance of $\mu_I$, that is $\Var(\mu_I):=\frac1{\sigma^2}\sum_{\lambda_i\in I}(\lambda_i^2-\lambda^2)^2$, is strictly related to the variance of the second chaos, as shown in the next proposition.
\begin{corollary}[Quantitative Berry's cancellation] \label{cor:berrycanc} Let $\phi=\phi_I$ and let $\lambda=\lambda(\phi_I), \sigma=\sigma(\phi_I),\e=\e(\phi_I)$ as above. 
Then,
    \be\label{eq:RWvar} 
\Var\tyu \frac{\m L_{\phi_I}(M)[2]}{\lambda}\uyt= \frac{s_{\n-1}}{2s_\n\sqrt{\n}}  \cdot 
\frac{1}{\sigma^2}\qwe 
\tyu\frac{1}{\sigma^2}\sum_{\lambda_i\in I}(\lambda_i^2-\lambda^2)^2\uyt\frac{1}{\lambda^4}+O_\n(\e)
\ewq
    \ee
In particular, if the field $\phi_I$ is homothetic ($\e=0$), the second chaos cancels if and only if $I$ contains at most one eigenvalue. Here, $O_\n(\e)$ denotes a quantity that is bounded by a dimensional constant $c_\n$ times $\e$.
\end{corollary}
\begin{proof}
Observe that the covariance function $C$ of the unit-variance normalization $f$ of a random wave $\phi_I$ (as in \cref{subsec:assjet}), is 
\be 
C(x,y)=\frac{(1+R(x,y))}{\sigma^2}\sum_{\lambda_i\in I}\f_i(x)\f_i(y),
\ee 
where $R(x,y)$ denotes a function with $\|j''_{x,y}R\|_{g}\le \e$.
From this, we deduce the inequality $\int_{M\times M}\|j_{x,y}''C\|^2_{g^f}\le \kappa_\n \sigma^{-2}(1+O_\n(\e))$, for some constant $\kappa_\n>0$ depending only on the dimension $\n$. Then, the thesis follows from \cref{thm:non-homotetic} and \cref{thm:homo24}, via straighforward computations.
\end{proof}
Our focus is on 
the monochromatic case, corresponding to $I=[\ell,\ell+\eta_\ell]$, with $\eta_\ell=o(\ell)$, which we will denote as $\phi_\ell$. 
There is a slight ambiguity in the literature about the name of such model: some references (\cite{Zel09,CH20,Keeler}) reserve the word \emph{monochromatic} only to the case $\eta_\ell=1$, while others (see \cite{canzani_MRWsurvey,canzani_sarnak}, or the monochromatic random band-limited functions of \cite{igorsurvey,sarnakwigmanCPAM,wigEBNRF}) admit that $\eta_\ell\to +\infty$. For clarity, we follow the former approach and call $\phi_\ell$ defined as above, a \emph{weakly Monochromatic Riemannian Random Wave (wMRW)}.
\begin{corollary}\label{cor:wMRW}
There is a dimensional constant $c_\n>0$ such that if $\phi_\ell$ is a \emph{wMRW}, then 
    \be\label{eq:RWvarmono} 
\Var\tyu \frac{\m L_{\phi_\ell}(M)[2]}{\ell}\uyt\le c_\n\cdot   \frac{\ell^{1-\n}}{\eta_\ell} \cdot \tyu O\tyu \frac{\eta_\ell^2}{\ell^2}\uyt+ \e(\phi_\ell) \uyt.
    \ee
\end{corollary}
\begin{proof}
In this case $\lambda(\phi_\ell)=\ell+o(\ell)$ and the term corresponding to the variance of $\mu_{I}$, is bounded by $(\eta_\ell^2+2\eta_\ell \ell)^2=O(\ell^2\eta_\ell^2)$. Finally, note that Weyl's law yields that $\sigma(\phi_\ell)^2=\Theta(\ell^{\n-1}\eta_\ell)$, for any \emph{wMRW}.
\end{proof}

\begin{remark}

Berry’s cancellation phenomenon refers to the unexpected small variance of nodal volumes in monochromatic random waves, which is a consequence of the cancellation of the second chaos component. This effect has been observed in all previously studied cases, all involving homothetic fields, but, so far, no general theoretical framework had been established (see the discussion at the end of \cite[Sec. 3.6]{igorsurvey}). To the best of the authors' knowledge, such cancellation has not been observed for general Riemannian random waves (e.g., \cite{CH20, Gass2020}). In \cite{igorsurvey}, it is conjectured that, under suitable assumptions, Berry's cancellation occurs if and only if the field
is monochromatic.

A deeper understanding of this phenomenon is crucial for accurately characterizing the variance and limiting distribution of nodal sets, particularly in contexts involving critical points and level sets. In particular, the vanishing of the projection of the nodal length onto the second Wiener chaos is attributed to intrinsic spatial symmetries, such as those present on the sphere \cite{Wig10}.
With \cref{cor:berrycanc} and \cref{cor:wMRW}, we show how the exact cancellation of the second chaos occurring on the sphere, the plane and other known settings, fits into a more general phenomenon that involves all random waves, where in general the cancellation is seen only in the asymptotic regime $\lambda \to +\infty$, $\e\to 0$, and it consists in the fact that in \cref{cor:wMRW} we have that
\be 
\Var\tyu \frac{\m L_{\phi_\ell}(M)[2]}{\ell}\uyt=o\tyu \frac{1}{\sigma(\phi_\ell)^{2}}\uyt,
\ee
whereas, for more general non-monochromatic regimes, one should expect a big-Oh, i.e., $O(\sigma(\phi_\ell)^{-2})$ from \cref{cor:berrycanc}. (Recall that $\sigma(\phi_\ell)^2$ is the dimension, see \cref{eq:phiI}.)

 This phenomenon is significant because it reveals how the geometric properties of the space influence the statistical behavior of nodal sets, and it provides insights into the fluctuations and distribution of these sets. 
\end{remark}
\begin{remark}
    
In \cite{Gass2025}, the author provides a spectral explanation of Berry's cancellation in the context of stationary Gaussian-subordinated random fields on $\R^\n$, demonstrating that it holds for a broader class of processes, beyond Gaussian, and that it extends to any local, possibly singular, functionals of Gaussian fields. 
To compare the findings of \cite{Gass2025} with the setting of the present paper, consider a smooth stationary and isotropic Gaussian field $F\colon \R^\n\to \R$ and an open subset $M\subset \R^\n$ with smooth boundary; take $f(X(v))=\delta_0(F(v))\|d_vF\|$, $\phi=1_{M}$ in \cite[Eq. (1)]{Gass2025} (here, one should justify that such distributional interpretation is allowed), and let $\phi_\lambda(\cdot):=F(\lambda(\cdot))$. Then, the field $\phi_\lambda\randin \mC^\infty(M)$ falls in the setting of \cref{subsec:assjet} and \cite[Eq. (1)]{Gass2025} becomes
\be\label{eq:Z} 
Z_\lambda(1_M)=\frac{1}{\lambda^{\d /2}}\int_{\lambda M}\delta_0(F(v))\|d_vF\|\ dv=\lambda^{\n /2}\cdot \frac{\m L_{\phi_\lambda}(M)}{\lambda},
\ee
Notice that up to linear changes of coordinates in $\R^\n$ and $\R$, one can assume that the random variables: $F(0),\de_1F(0),\dots,\de_\n F(0)$ are i.i.d. standard Gaussians, so that the field $\phi_\lambda$ will have $\e=0$
eccentricity and
average frequency $\lambda$ (in the sense of \cref{def:three}).
In the same setting, in \cite{zcpcalt2025AnconaGassLetendreStecconi}, the authors obtain moment asymptotics and a central limit theorem for the quantity $\nu_{F}(\lambda M)=\m L_{\phi_\lambda}(M) / \lambda$. 
\end{remark}
\subsubsection{Benchmark example 1: Berry's field}
The reason for the name \emph{monochromatic} is that $\phi_\ell$ can be compared with Berry's field $\berry$ on $\R^\n$, which is almost surely a solution of the equation: $\Delta \berry = -\berry$ (see \cite{Berry1977,NourdinPeccatiRossi2017}). Berry's field $\berry \randin \mC^\infty(\R^\d)$, is defined as the stationary and isotropic Gaussian field on $\R^\n$ with covariance function $\E\berry(x)\berry(x+u)=
\fint_{S^{\n-1}} e^{i \langle 
u, \theta \rangle}d\theta.$\footnote{Berry's field is often defined with a different constant factor in front of it. We chose the normalization such that $\sigma(\berry)=1$.} The restriction of $\berry$ to any compact domain $B\subset \R^\n$ is homothetic with frequency one:
\be 
\e(\berry)=0; \quad \sigma(\berry)=1; \quad 
 \lambda(\berry)=1.
\ee
\subsubsection{Benchmark example 2: Gaussian spherical harmonics}
On the sphere $S^\n$, the interval $I=[\ell,\ell+o(\ell)]$, for large $\ell$ (see also \cite{Tod24}) contains the square root of only one eigenvalue: $\lambda_\ell^2:=\ell(\ell+\n-1)$; therefore $\phi_\ell=\phi_{\kop\lambda_\ell\pok}=:T_\ell$ is a Gaussian spherical harmonics of degree $\ell$, the same considered for instance in \cite{Wig09, marinucci2023laguerre}.
\be 
\e(T_\ell)=0;
\quad 
\sigma(T_\ell)=\frac{2\ell+\n-1}{\ell}\binom{\ell+\n-2}{\ell-1};\quad  \lambda(T_\ell)=\sqrt{\ell(\ell+\n-1)}.
\ee
\subsubsection{Scaling limit}
The latter two models serve as a term of comparison to study the more general situation of $\phi_\ell$ on an arbitrary manifold, combined with the following heuristic: for any $x\in M$ and given an isometry $\theta_x\colon  (T_xM,g_x)\to (\R^\n,\mathbbm{1})$, we have that
\be\label{eq:scalim} 
\frac{1}{\sigma(\phi_\ell)}\phi_\ell\tyu\exp_x\tyu \ell^{-1} u \uyt \uyt
\xrightarrow[\ell\to \infty]{}
\berry\tyu \theta_x(u)\uyt,
\ee
where the convergence is in distribution in the space $\mC^\infty(T_xM)$ of smooth functions of $u$. In fact, the limit \eqref{eq:scalim} holds under some additional condition on the metric, or on $\eta_\ell$, see \cite{CH20,Gass2020,Keeler,igorsurvey}. This implies that if $f_\ell$ is the unit variance noramlization of $\phi_\ell$, then
\be 
g^{f_\ell}=\frac{\ell^2}{\n} \tyu g+o(1) \uyt,
\ee
as $\ell\to+\infty$.
In such case, from Weyl's law and \cref{eq:phiI} we have that
\be \label{eq:phiell}
\e(\phi_\ell)\to 0;
\quad 
\sigma(\phi_\ell)^2=\Theta(\ell^{\n-1});\quad \lambda(\phi_\ell)^2=\frac1{\sigma^2}\sum_{\lambda_i\in [\ell,\ell+o(\ell)]}\lambda_i^2 = \Theta\tyu \ell^2\uyt.
\ee
This explains our choice of focusing on the case of approximately homothetic fields.


 \subsection{Related results} 
As highlighted in \ref{sec:MRW}, the nodal structure of Riemannian Random waves, 
introduced  in \cite{Zel09} and recalled in \cref{eq:RW}, has drawn particular attention in the last decades. We refer to the recent surveys \cite{igorsurvey,canzani_MRWsurvey} and the references therein.
In particular, in \cite{Ber85} and \cite{Zel09}, the expected nodal length 
was computed when
$I=[0,\ell]$, while in \cite{Zel09}  the case $I=[\ell,\ell+1]$ is also addressed.
More precisely, in the latter case,
Zelditch \cite{Zel09}
demonstrated that the total expected nodal $(\n-1)-$volume is asymptotic to
$$\mathbb{E}[\vol{\n-1}(\phi_I^{-1} (0))]=C_M \ell,$$
where $C_M > 0$ is an explicit constant proportional to the Riemannian volume of $M$, under the assumption that $M$ is Zoll or aperiodic. \cite{Gass2020} extended this to arbitrary manifolds, using a larger window regime $I=[\ell-\ell^\tau,\ell]$, proving that the variance tends to zero. Under generic assumptions on $M$, an upper bound for the variance of order
$O(\ell^{2-(n-1)/2})$ has been established in \cite{CH20}. In dimension $\n=2$, the nodal length of the restriction to a shrinking ball has been studied more closely by \cite{Dierickx}, proving a central limit theorem.

The study of Gaussian Laplacian eigenfunctions has been explored in some specific settings as mentioned also in \cref{sec:homotheticfields}. In particular, when $M=\mathbb{S}^n$,  
an upper bound for the variance of the nodal volume of random spherical harmonics was derived in \cite{Wig09}, which was subsequently improved in \cite{marinucci2023laguerre}. Special attention has been devoted to the two-dimensional case. In \cite{Wig10} a precise formula for the variance of the nodal length was given, proving Berry’s cancellation phenomenon, while a quantitative Central Limit Theorem was later established in \cite{MRW20}. 
Further investigations on the sphere include the nodal length inside shrinking geodesic balls of radius slightly above the Planck scale, studied in \cite{Tod20}, moderate deviation estimates for nodal lengths on the entire sphere and on shrinking domains,  discussed in \cite{MacciRT}; and the analysis of band-limited functions for intervals $I=[\ell,\ell+o(\ell)]$, in \cite{Tod24}.
The nodal volume of Berry’s random fields, originally introduced in \cite{Berry1977, Berry2002}, has also been the subject of extensive study. 
For results concerning nodal lines and nodal length in three-dimensional space, we refer to \cite{DalmaoEstradeLeon, Dalmao}. Planar models have also received significant attention: the nodal length of planar random waves is investigated in \cite{NourdinPeccatiRossi2017, Vidotto}; spatial functional limit theorems for the zero sets are established in \cite{NPV23}. Further developments include the study of nodal length restricted to compact subsets in \cite{PeccatiVidotto2020}, and the analysis of fluctuations in the nodal number of the complex planar Berry Random Wave Model in \cite{Smutek}.
Finally, lot of results have been established for the nodal volume of Arithmetic random waves - eigenfunctions of the Laplacian on flat tori.
 In \cite{RudWig08}, the volume of the nodal sets was analyzed, while the fluctuations of the nodal length in dimension two are investigated in \cite{WigmanAnnMath}. 
 A quantitative non-central limit theorem for the nodal length was established in \cite{MPRW}, and its behavior in shrinking domains was examined in \cite{BMW}. Additionally, the moderate deviation principle for the nodal length has been studied both on the entire manifold and on shrinking toral domains in \cite{MacciRV}. The distribution of the nodal area and its chaotic expansion were investigated in \cite{BM19, Cammarota2017NodalAD}. Higher-dimensional extensions were explored in \cite{ORW08}.
 
 Many of these results rely on the Wiener chaos decomposition in \cref{eq:chaosexpVALE}, see also the survey \cite{Rossi2018RandomNL}, which turned out to be a powerful analytical tool for studying nodal volumes of Gaussian random fields on manifolds, although it is essentially limited to those settings that in this paper are called homothetic. However, its classical formulation is computationally expensive in high dimensions as (\ref{eq:chaosexpVALE}) suggests. Consequently, the literature has actively explored methods for reducing this complexity. 
Early works \cite{Tha93, Koc96} provided explicit Wiener-Ito chaos expansion coefficients for functions of the form $F(x)=f_0(||x||)P(x)$ on $\mathbb{ R}^n$, with $P$ a harmonic polynomial, using Laguerre polynomials. This was recently extended by \cite{Not23} to matrix-Hermite expansion of $F(X)=f_0(XX^T)$ for Gaussian matrices X under additional assumptions (such as independent columns), which can be applied to nodal intersections for at most 3 i.i.d. copies of arithmetic random waves in $\R^3$. Independently, \cite{marinucci2023laguerre} derived a similar compact formula involving Laguerre polynomials for the nodal volume of hyperspherical harmonics. 

The nodal volume of Gaussian fields have a prominent place in the algebraic context, in the study of Gaussian ensembles of random polynomials, such as Kostlan polynomials, see \cite{ShSm1,letendre_varvolII_2019,LETENDRE2016,RootsKost_ancona_letendre,letendre_varvolI_2019}. In particular, \cite{letendre_varvolII_2019} obtains precise variance asymptotics using chaos decomposition.

Finally, we mention a series of works concerned with the nodal volumes for general Gaussian fields: \cite{ArmentanoAzaisLeonMordecki, Anc21c, Gas21b,Gas21t, fom2024GassStecconi,letendre2023multijet} analyze the moments and cumulants; \cite{KRStec,ersz2024MathisStecconi, Nicolaescu, NicSavale, DangRiv2018, geospin2022LMRStec} are concerned with the expectation of the nodal volume of Gaussian sections of vector bundles; and  \cite{nvdfg2024PeccatiStecconi, AngstPoly} achieved the differentiability and absolute continuity of their distributions using Malliavin calculus techniques. 

\section{The chaos expansion of the nodal volume}
\subsection{Chaos expansions}\label{sec:chaos} Let $H_q$ be the $q^{th}$ Hermite polynomial, defined by 
\be\label{eq:H}
\sum_{q\in \N}H_q(x)\frac{t^q}{q!}=e^{tx-\frac{t^2}{2}}.
\ee
Let $\xi=(\xi_1,\dots,\xi_\d)^T\sim N(0,\mathbbm{1}_\d)$ be a Gaussian random vector in $\R^\n$.
By definition, the $q^{th}$ chaotic component 
$F(\xi)[q]$ 
 of a random variable $F(\xi)\in L^2$ is its $L^2$ projection onto the closed subspace of $L^2$ generated by the random variables $\kop H_q(\langle\xi,v\rangle): v\in S^{\d-1}\pok$, i.e., the space
\be\label{eq:Wchaos}
W_q[\xi]:=\mathrm{span}\kop H_q(\langle \xi,v\rangle): v\in S^{\d-1}\pok.
\ee
Notice that the span above is already closed. The importance of Hermite polynomials in the analysis of Gaussian random variables is due to their orthogonality relations, expressed in the following formula for which we refer to \cite[Proposition 2.2.1]{Nourdin_Peccati_2012}.
\begin{lemma}\label{lem:orto}
Let $\gamma_1,\gamma_2\sim N(0,1)$ be jointly Gaussian random variables. Then for any $q,q'\in \N$, we have
\be 
\E\kop H_q(\gamma_1)\cdot H_{q'}(\gamma_2)\pok=\kop 
\begin{aligned}
    q!\tyu \E\kop \gamma_1\gamma_2\pok\uyt^q,& \quad \text{if $q= q'$}
    \\
    0,& \quad \text{if $q\neq q'$}.
\end{aligned}\right.
\ee
\end{lemma}
\begin{remark}
An orthogonal basis of $W_q[\xi]$ is given by the 
collection of all products 
$H_\a(\xi):=\prod_{i=1}^\d H_{\a_i}(\xi_i)$, 
with $\a\in \N^\d$ multi-indices of degree $|\a|=\a_1+\dots +\a_\d=q$. 
Thus, we know that $W_q[\xi]$ has finite dimension $ \binom{q+\d-1}{q}$. However, 
we want to use a different representation of chaotic components, closer to Definition \eqref{eq:Wchaos}.
\end{remark}
\subsection{A Chaos expansion of a chi-variable}
Let $\xi=(\xi_1,\dots,\xi_\d)^T\sim N(0,\mathbbm{1}_\d)$ be a Gaussian random vector (representing the gradient $\nabla T(x)$ at some given point, in some given coordinate chart $\R^\d\cong T_xM$ of a Gaussian field $T$ defined over the sphere $M=S^\d$, or some other manifold $M$.)
\subsubsection{The norm of a vector is the average norm of a projection}
\begin{lemma}\label{lem:circleavg1}
Let $\xi \in \R^\n$ be any vector, then
\be 
\|\xi\|=\frac{\pi}{s_\n}\int_{S^{\n-1}}|\langle \xi,v\rangle | dv.
\ee
\end{lemma}
\begin{proof}
Clearly, the integral depends only on the norm of $\xi$, by rotational invariance, hence
\be 
\int_{S^{\n-1}}|\langle \xi,v\rangle | dv=\|\xi\|\int_{S^{\n-1}}|v_1 | dv=\|\xi\|\beta(\n,1);
\ee
where $\beta(\n,1)$ is a dimensional constant, computed in \cref{lem:beta} as $\beta(\n,1)=\frac{s_\n}{\pi}$.
\end{proof}
\begin{theorem}\label{thm:genchicaos}
Let $\xi\sim \m N(0,G)$ be a Gaussian random vector in $\R^\n$, with non-degenerate covariance matrix $G$. Then, the $q^{th}$ chaotic component of $\|\xi\|$ vanishes if $q$ is odd and for $q$ even is
\be\label{eq:radialHer}
 \tyu \|\xi\|\uyt[q]=A(\n,q)\int_{S^{\n-1}}H_q\tyu \frac{\langle \xi,v\rangle }{\sqrt{v^TGv}} \uyt \sqrt{v^TGv} \ dv,
\ee
where 
\be \label{A}
A(\n,2b)=\frac{\pi}{s_\n}\frac{(-1)^{b-1}}{2^{b-1}\sqrt{2\pi}(2b-1)b!},
\ee
for all $b\in \N$. In particular, when $G=\mathbbm 1_\n$ the above formula yields the chaotic components of a chi random variable of parameter $\n$.
\end{theorem}
\begin{proof}
We start by expressing $\|\xi\|$ as a circle average, with the help of \cref{lem:circleavg1}. 
\bega 
\|\xi\|&=\frac{\pi}{s_\n}\int_{S^{\n-1}}|\langle \xi,v\rangle | dv
=
\frac{\pi}{s_\n}\int_{S^{\n-1}}\frac{|\langle \xi,v\rangle |}{\sqrt{v^TGv}} \sqrt{v^TGv} \cdot dv.
\eega
Now, observe that $\frac{\langle \xi,v\rangle}{\sqrt{v^TGv}}$ is a unit variance Gaussian variable. It follows that it is enough to prove the theorem in dimension $1$, for $\gamma\sim \m N(0,1)$. The chaotic expansion of a chi random variable is well known (see for instance \cite{MPRW,Cammarota2017NodalAD} and we report a short proof adapted in appendix \cref{lem:chi}) and is
\be 
|\gamma|=\sum_{b\in \N} c_\chi(2b)H_{2b}(\gamma), \quad \text{with} \quad c_\chi(2b)=\frac{(-1)^{b-1}}{2^{b-1}\sqrt{2\pi}(2b-1)b!}.
\ee
Threfore, setting $A(\n,q):=\frac{\pi}{s_\n}c_\chi(q)$ concludes the proof.
\end{proof}
\subsection{Proof of  \texorpdfstring{\cref{thm:nodalchaos}}{Theorem}}\label{sec:proof_th1.2}
As in \cite{kratzleon, MRW20, MPRW,Cammarota2017NodalAD}, we obtain the expansion of 
\be 
\m L_f(dx):=\delta_0(f(x))\|\nabla_xf\|dx
\ee
as a product of the expansion of $\|\nabla_xf\|$, obtained with \cref{thm:genchicaos} and of that of $\delta_0(f(x))$, expressed in the following lemma.
\begin{lemma}\label{lem:delta} For $\gamma\sim N(0,1)$, the chaos expansion of the distribution $\delta_0$ \emph{makes sense} and is equal to:
\be 
\delta_0(\gamma)=\sum_{a\in \N} H_{2a}(\gamma)\frac{(-1)^a}{2^aa!\sqrt{2\pi}}.
\ee
\end{lemma}
\begin{proof}
By \cite[Eq. (3.28)]{MPRW}, we have that
\be 
\delta_0(\gamma)=\sum_{a\in \N} H_{2a}(\gamma)\frac{H_{2a}(0)}{(2a)!\sqrt{2\pi}}, \quad \text{and} \quad H_{2a}(0)= (-1)^{a}\frac{(2a)!}{2^a a!}.
\ee
\end{proof}
\begin{proof}[proof of \cref{thm:nodalchaos}]
Observe that in an orthonormal basis of $T_xM$, the covariance matrix $G$ of $\nabla_x f $ is, by \cref{eq:assjet}, such that
\be 
v^TGv=\E\kop |\langle\nabla_xf,v\rangle|^2\pok=g^f_x(v,v),
\ee
for any $v\in S(T_xM)$.
Therefore, \cref{thm:genchicaos} yields
\bega 
&\vol{\d-1}\tyu f^{-1}(0)\uyt 
= 
\int_M \delta_0(f(x)) \|\nabla_x f\| dx
\\
&=
\int_M 
\tyu \sum_{a\in \N} H_{2a}(f(x))\frac{(-1)^a}{2^aa!\sqrt{2\pi}}\uyt
\tyu \sum_{b\in \N} A(\n,2b)\int_{S(T_xM)}H_{2b}\tyu \frac{\langle \nabla_xf, v\rangle}{\sqrt{g^f(v,v)}}\uyt \sqrt{g^f(v,v)}dv\uyt
dx
\\
&= \sum_{q \in \mathbb{N}}
\sum_{2(a+b)=q}\tilde{\coeff}(\n,a,b)\int_M\int_{S(T_xM)} 
 H_{2a}(f(x))
 H_{2b}\tyu \frac{\langle \nabla_xf, v\rangle}{\sqrt{g^f(v,v)}}\uyt \sqrt{g^f(v,v)}dvdx
\eega
so that 
\be 
\tilde{\coeff}(\n,a,b)=\frac{(-1)^a}{2^aa!\sqrt{2\pi}}\cdot A(\n,2b)=\frac{(-1)^a}{2^aa!\sqrt{2\pi}}\cdot \frac{\pi}{s_\n}\frac{(-1)^{b-1}}{2^{b-1}\sqrt{2\pi}(2b-1)b!},
\ee
thus we conclude.
\end{proof}
\subsection{The pointwise frequency endomorphism}\label{sec:pointfreq}
 One natural way of comparing the two metrics $g$ and $g^f$ is to express the eigenvalues $a_1(x)^2,\dots,a_\n(x)^2$ of $g^f_x$ with respect to $g_x$, which are locally smooth functions of $x \in M$, defined up to order\footnote{Any local trivialization of the tangent bundle $TM$ on some open set $U\subset M$ implies that, there is a permutation $\sigma\in \Sigma_\d$, such that $a_{\sigma(i)}$ is a smooth functions on $U$, for each $i\in \{1,\dots,\d\}$.}. 
This perspective has been considered in \cite{elk2024PistolatoStecconi,bierme2D3D}. 
The object that completely encodes the comparison between the two metrics is the linear change of variables $\Lambda_x:T_xM\to T_xM$, 
\be 
\Lambda_x:=\sqrt{ \n g_x^{-1}g_x^f}, \footnote{
The formula $\Lambda=\sqrt{\n g^{-1}g^f}$ is true for the matrices. Let $\Lambda,g$ and $g^f$ be the matrices in any given basis. By the spectral theorem, there exists an invertible matrix $A$, such that $A^Tg A$ and $A^Tg^fA$ are both diagonal. Then, 
\be \label{eq:Lambda}
\Lambda=A\qwe \n \tyu A^Tg A\uyt^{-1}\tyu A^Tg^f A\uyt \ewq^{\frac12}A^{-1}= \tyu \n g^{-1}g^f\uyt^{\frac12},
\ee 
by the definition of the square-root of a positive definite symmetric matrix.
}
\ee
whose eigenvectors are an orthogonal basis for both $g^f$ and $g$ and whose eigenvalues are $\sqrt{\n}a_1(x),\dots, \sqrt{\n}a_\n(x)$. Thus, $\frac 1\n\Lambda_x$ is an isometry from $(T_xM,g_x^f)$ to $(T_xM,g_x)$, meaning that 
$ 
\frac1\n{g_x(\Lambda_xu,\Lambda_x v)= g^f_x(u,v)}
$. We characterize $x\mapsto \Lambda_x$ as follows. 
\begin{definition}
For every $x\in M$, we define 
$\Lambda_x\colon T_xM\to T_xM$ as the unique self-adjoint positive definite endomorphism of $(T_xM,g_x)$ such that $g^f_x=\frac{1}{n}\Lambda_x^*g_x$. We call the resulting tensor $\Lambda=\Lambda^f$ the \emph{pointwise frequency endomorphism} of $f$.
\end{definition}
Globally, $\Lambda$ is a $\mC^1$ section of $\mathsf{GL}(TM)$ --- the bundle whose fiber is the group of linear automorphisms\footnote{More precisely, $\Lambda_x$ belongs to the sub-bundle of those automorphisms that are diagonalizable and have positive eigenvalues.} of $T_xM$--- and it allows to compare the two metrics $g^f=\frac{1}{\n}\Lambda^*g$ and $g$, in that ${\Lambda}$ sends an orthonormal basis of $g^f$ to one of $\frac1\n g$.
\subsubsection{The frequency} The pointwise frequency $\lambda(f,x)$ defined in \cref{def:pointwisefreq} can be interpreted as an average of the eigenvalues of $\Lambda_x$, indeed
\be
\lambda(f,x)^2=\frac1\n\tr\tyu \Lambda_x^2\uyt=\tr_{g}\tyu g^f_x\uyt=\E\kop \|d_xf\|^2\pok=-\E\kop f(x)\cdot \Delta f(x)\pok, 
\ee 
the last identity involves the Laplace-Beltrami operator $\Delta$ of $(M,g)$ and it is easily obtained by differentiating twice the identity $\E|f(x)|^2=1$ and taking the trace with respect to $g$. 
Evidently, if $f$ satisfies the eigenvalue identity $\Delta f=-\lambda^2 f$ almost surely, then $\lambda=\lambda(f)$. 

The average frequency, see \cref{def:three}, is obtained by averaging over $M$
\be \label{eq:lambda}
\lambda(f)^2=\fint_M\lambda(f,x)^2dx.
\ee
Let $\lambda=\lambda(f)$. By \cref{def:three}, the eccentricity $\e=\e(f)$ quantifies the error in approximating $\Lambda_x$ with a homothety, that is, 
\be\label{eq:LambdaApprox}
\frac{1}{\lambda}\Lambda_x =  \id_{T_xM}+O\tyu \e\uyt,
\ee
regardless of its frequency. 
The intuition behind this is guided by the case of Riemannian Random Waves $\phi_\lambda$, see \cref{sec:MRW}.
Indeed, note the Berry random field $\berry$ on $\R^{\d}$ has Adler-Taylor metric $g^{\berry}=\frac{1}{\d}\mathbbm{1}_{\d}$, $\Lambda^\berry=\mathbbm{1}_{\n}
$ and $\lambda(\berry)=1$. 
\subsubsection{Alternative expressions}

\begin{lemma}\label{lem:sphChange}
Let $L\colon \R^\n\to \R^\n$ be a linear isomorphism. Define $\phi\colon S^{\n-1}\to S^{\n-1}$ as $\phi(u)=\|Lu\|^{-1}Lu$. Then, $\phi$ is a diffeomorphism with differential $D_u\phi\colon u^\perp \to (Lu)^\perp$ and Jacobian determinant being
\be\label{eq:sphChange}
D_u\phi= \frac{\prod_{(Lu)^\perp}\circ L|_{u^\perp}}{\|Lu\|},\qquad \det\tyu D_u\phi\uyt=\frac{\det (L)}{\|Lu\|^\n},
\ee
respectively.
\end{lemma}
\begin{proof}
Straighforward.
\end{proof}
From this lemma, it follows that we can apply changes of variables $v=\phi(u)$ in the integrals over a sphere, using the following formula.
\be 
\int_{S^{\n -1}}h\tyu \frac{v}{\|L^{-1} v\|}\uyt dv=\int_{S^{\n -1}}h\tyu Lu\uyt \frac{\det (L)}{\|Lu\|^\n} du.
\ee
This allows us to obtain alternative expressions for \cref{eq:nodalchaosAT} in terms of $\Lambda$ and $\lambda$.
\begin{corollary}\label{cor:altexprL}
Let the setting of \cref{thm:nodalchaos} prevail. Let $\Lambda$ and $\lambda$ be defined as in \eqref{eq:Lambda} and \eqref{eq:lambda}, respectively. Then,
\begin{gather}
\m L_f[q] = 
\\
\sum_{a,b\in \N,\ a+b=\frac{q}2}\frac{\coeff(a,b)}{s_\n\sqrt{\n}}\int_M\int_{S(T_xM)} 
 H_{2a}(f(x))
 H_{2b}\tyu \frac{\langle \nabla_xf, v\rangle}{\|\Lambda_x v\|} \sqrt{\n}\uyt \left\|\Lambda_xv\right\| dv dx= 
 \label{eq:nodalchaosL}
 \\
\sum_{a,b\in \N,\ a+b=\frac{q}2}\frac{\coeff(a,b)}{s_\n\sqrt{\n}}\int_M\int_{S(T_xM)} 
 H_{2a}(f(x))
 H_{2b}\tyu \langle \Lambda_x^{-1}\nabla_xf, v\rangle \sqrt{\n}
 \uyt 
 \frac{\|\Lambda_x^{-1}v\|^{-(\n+1)}}{\det (\Lambda_x)}dv dx, 
 \label{eq:nodalchaosLinv}
\end{gather}
where $\coeff(a,b)$ is defined in \cref{eq:defC}. 
In particular, in the homothetic case, that is, if $\Lambda_x=\lambda \id_{T_xM}$, all  three expressions reduce to
\be\label{eq:nodalchaosHomo}
\m L_f[q]=\lambda \sum_{a,b\in \N,\ a+b=\frac{q}2}\frac{\coeff(a,b)}{s_\n\sqrt{\n}}\int_M\int_{S(T_xM)} 
 H_{2a}(f(x))
 H_{2b}\tyu \langle \nabla_xf, v\rangle \frac{\sqrt{\n}}{\lambda}
 \uyt 
 dv dx.
 \ee
\end{corollary}
\begin{proof}
The first formula is obtained from \cref{eq:nodalchaosAT}, by using the definition of $\Lambda_x$ as
\be 
\|v\|_{g^f}^2=g^f_x(v,v)=\frac{1}{\n}\|\Lambda_x v\|^2
\ee
The second is obtained from \cref{eq:nodalchaosL} by applying \cref{lem:sphChange} with $L=\Lambda_x^{-1}$, in the integrand, for every $x\in M$, see \cref{eq:sphChange}. The third is obvious.
\end{proof}
\subsection{The bivariate chaotic measures}
In what follows, it is going to be convenient to have a notation for the integrands in \cref{eq:nodalchaosL}, which we will refer to as the \emph{bi--chaotic components} of $\lf$, defined as follows.
\begin{definition}[bi--chaotic components]
    Let the assumptions of \cref{cor:altexprL} prevail. For every pair $a,b\in \N$, we define the (absolutely continuous) measure $\lfs[a,b]$ on $S(TM)$ via the expression:
    \be \label{luxab}
\lfs(dudx)[a,b]=\frac{\coeff(a,b)}{s_\n\sqrt{\n}} 
 H_{2a}(f(x))
 H_{2b}\tyu \frac{\langle \nabla_xf, u\rangle}{\|\Lambda_x u\|} \sqrt{\n}\uyt \left\|\Lambda_x u\right\| du dx.
    \ee
Similarly, we define the measure $\lfs[q]$ on $S(TM)$ as 
\be \label{luxq}
\lfs(dudx)[q]=\sum_{a,b\in \N,\ a+b=\frac{q}2}
\lfs(dudx)[a,b].
\ee
Moreover, we define the measures $\lf[a,b]$ and $\lf[q]$ on $M$ as the push-forward measures via the canonical projection $\pi\colon S(TM)\to M$, of $\lfs[a,b]$ and $\lfs[q]$, meaning:
\be \label{lxab-xq}
\lf(dx)[a,b]:=\int_{S(T_xM)}\hat{\lf}(dudx)[a,b] \quad \text{and}\quad \lf(dx)[q]:=\int_{S(T_xM)}\hat{\lf}(dudx)[q].
\ee
We also recall the definition of the \emph{nodal measure} $\lf$ on $M$ 
\be 
\lf(dx)=\delta_0(f(x))\|d_xf\|dx,
\ee
that is a random measure on $M$ defined as in \cref{coarea}.
\end{definition}
Of course, these measures are absolutely continuous with respect to the canonical Riemannian measure $dudx$ on $S(TM)$, with a smooth density, so they can alternatively be thought as smooth functions. 

\cref{cor:altexprL} says that what we denoted as $\lf[q]$ is indeed the $q^{th}$ chaotic component of the nodal measure $\lf$.
\subsection{Non-centered Gaussian fields or non-zero level}
We can extend the definition of the measure $\hat{\m L}$ to non-centered field of the form $f-t$, for some $t\in \R$, whose nodal set is the level set $f^{-1}(t)$. 
\be 
\hat{\m L}_{f-t}(dudx)[a,b]:=
\frac{e^{-\frac{t^2}{2}}H_{2a}(t)}{H_{2a}(0)}\hat{\m L}_{f}(dudx)[a,b], \quad
\ee 
and
\begin{gather} \label{eq:luxabmean}
 \hat{\m L}_{f-t}(dudx)[q]:=\sum_{a+b=\frac{q}{2}}\hat{\m L}_{f-t}(dudx)[a,b]
=
\hat{\m L}_{f-t}(x,u)[q] du dx.
    \end{gather}
where $\hat{\m L}_{f-t}(x,u)[q]$ is the function  defined on $(x,u)\in S(TM)$ as
\be 
\hat{\m L}_{f-t}(x,u)[q]:=
\sum_{a+b=\frac{q}{2}}
e^{-\frac{t^2}{2}}\frac{H_{2a}(t)}{H_{2a}(0)}
\frac{\coeff(a,b)}{s_\n\sqrt{\n}} H_{2a}(f(x))H_{2b}\tyu \frac{\langle \nabla_xf, u\rangle}{\|\Lambda_x u\|} \sqrt{\n}\uyt \left\|\Lambda_x u\right\|.
\ee
We will see with \cref{cor:cosine} below that this is indeed the right expression for the chaotic decomposition of general levels.
    \subsection{Cosine transform integrals}\label{sec:cosine}
We defined $\lfs(dxdu)[q]$ as a measure on the total space of the sphere bundle $S(TM)$ with density $\lfs(x,u)[q]$.  In this subsection we give an indication of what they represent. 
The key object to understand them is the \emph{cosine} transform (see \cite{bible}):
\be 
\mathrm{H}: \m{M}(S^{\n-1})\to \mC^0(S^{\n-1}), \quad  \mu\mapsto h, \quad \text{s.t. }\quad h(x)=\int_{S^{\n-1}}|\langle x,u \rangle |\mu(du),
\ee
transforming a finite signed measure $\mu$ on $S^{\n-1}$ into a continuous function $h=\mathrm{H}(\mu)$. For instance, in the proof of \cref{thm:nodalchaos}, we leveraged the fact that the cosine transform of the volume measure $\mu=\vol{\n-1}$ is a constant $\mathrm{H}(\vol{\n-1})=\beta(\n,1)$ (see \cref{lem:circleavg1}). The map $\mathrm{H}$ can be defined on any metric vector space $(T,g)$, in particular, on each tangent space $(T_xM,g_x)$ of a Riemannian manifold. \cite[Theorem 2.26]{MathisZA} gives an account of its properties, one of which is that $\mathrm{H}$ is sequentially continuous with respect to the weak-$*$ convergence. 
The image of the cosine transform is a special subset $\mathrm{H}(\R^n)\subset \mC^0(S^{\n -1})$ consisting of support function of generalized zonoids (see \cite{bible,MathisZA}). We can thus observe that our proof automatically extends \cref{thm:nodalchaos} to all integrals over the nodal set of $f$ that depend on $df$ through a cosine transform. We prove now \cref{eq:cosine}.
\begin{corollary}\label{cor:cosine}
    Let the assumptions of \cref{thm:nodalchaos} prevail. Let $h\colon S(TM)\to \R$ be a function 
    such that for any $x\in M$, there exists a signed measure $\mu_x$ on $S(T_xM)$ such that $\mathrm{H}(\mu_x)=h|_{T_xM}$ and such that $x\mapsto \mu_x$ is continuous. Then $h$ is continuous and  
    \be 
\tyu \int_{f^{-1}(t)}h\tyu x,\frac{d_xf}{\|d_xf\|}\uyt dx \uyt [q]
=
\int_{M}\int_{S(T_xM)} \hat{\m L}_{f-t}(dxdu) [q]\cdot \mu_x(du) dx,
    \ee
    where $\hat{\m L}_f(dxdv) [q]$ is the measure defined in \cref{eq:luxabmean} and $t\in \R$.
\end{corollary}
\begin{proof}
For the continuity part, let us pass to a local chart over an open subset $O\subset M$, so that $S(TO)\equiv O\times S^{\n-1}.$
By \cite[Theorem 2.26]{MathisZA} we have that $\mathrm{H}$ is sequentially continuous with respect to the weak convergence of signed measures, that is, the weak-$*$ topology obtained by the identification $\m M(S^{\n-1})=\mC^0(S^{\-n-1})^*$. This implies that the function $x\mapsto h(x,\cdot)$ is sequentially continuous from $M\to \mC^0(S^{\-n-1})$; hence it is continuous, being it a function between metrizable spaces. Finally, the continuity of $x\mapsto h(x,\cdot)$ is equivalent to the continuity of $h$ in both variables.

Knowing that $h$ and $\mu$ are continuous enables us to exchange the order of the integrals and of the chaotic projection in what follows.
We can repeat the exact same arguments employed in the proof of \cref{thm:nodalchaos}, adding the function $h$ in the integral:
    \bega
\int_{f^{-1}(t)}h\tyu x,\frac{d_xf}{\|d_xf\|}\uyt dx [q]
&=\int_{M}\delta_t(f(x))\|d_xf\| h\tyu x,\frac{d_xf}{\|d_xf\|}\uyt dx
\\
&=\int_{M}\int_{S(T_xM)}\tyu \delta_t(f(x)) |\langle d_xf, u \rangle| \uyt \mu_x(du) dx
\\
&=\int_{M}\int_{S(T_xM)} \hat{\m L}_{f-t}(x,u) [q] \mu_x(du) dx.
    \eega
For general level $t$, it is sufficient to notice that the chaos component of the delta distribution satisfies: $\delta_t(\gamma)[2a]=e^{-\frac{t^2}{2}}\frac{H_{2a}(t)}{H_{2a}(0)}\delta_0(\gamma)[2a]$, see \cref{lem:delta}.
\end{proof}
\section{The Variance of the nodal volume}
\subsection{Diagram formula for Hermite-Laguerre polynomials}
Actually, we will only write a formula for the polynomials $S_{\d,q}(\xi)=\int_{S^{\d-1}}H_q\tyu \langle x, v\rangle\uyt dv$. This is equivalent to work with Laguerre polynomials, due to \cref{eq:LaS}
, after having computed the constants $c_{\d,q}$.
\begin{lemma}\label{lem:diagfour}
Let $\gamma_1,\gamma_2,\gamma_3,\gamma_4\sim N(0,1)$ with correlations $\E\gamma_i \gamma_j=C_{ij}$ and assume that $C_{12}=C_{34}=0$. Then,
\be 
\E\kop H_a(\gamma_1)H_b(\gamma_2)H_{a'}(\gamma_3)H_{b'}(\gamma_4)\pok= \left\{ \begin{aligned}
    & a!b!a'!b'!\sum_{k=\max\{ 0, a'-b\}}^{\min\{a,a'\}} \frac{C_{13}^kC_{14}^{a-k}C_{23}^{a'-k}C_{24}^{b-a'+k}}{k!(a-k)!(a'-k)!(b-a'+k)!}
    \\
    &0\quad \text{ if $a+b\neq a'+b'$.}
\end{aligned}
\right.
\ee
\end{lemma}
\begin{proof}
We can use the diagram formula, in the form reported in \cite[Section 5.2]{marinucci2023laguerre} and \cite[Lemma 5.2]{CARAMELLINO2024110239}. Let $\mathcal{K}:=\mathcal{A}(a,b,a',b')\subset \N^{4\time 4}$ be the subset of $4\times 4$ matrices $K=\{K_{i,j}\}_{1\le i,j\le 4}$ with entries in $\N$, such that $K_{ii}=0$ for all $i$, $K^T=K$ and $(1,1,1,1)K=(a,b,a',b')$. Then,
\be 
\E\kop H_a(\gamma_1)H_b(\gamma_2)H_{a'}(\gamma_3)H_{b'}(\gamma_4)\pok=a!b!a'!b'! 
\sum_{K\in \mathcal{K}}\prod_{1\le i<j\le 4} \frac{C_{ij}^{K_{ij}}}{K_{ij}!},
\ee
where we adopt the convention that $0^0=1$ and, of course, $0^k=0$ for all $k\ge 1$. Thus, since $C_{12}=C_{34}=0$, the  terms in the sum with $K_{12}\neq 0$ vanish, as well as those with $K_{34}\neq 0$. This allows us to let the sum run only over the symmetric matrices $K$ of the form 
\be 
K=\begin{pmatrix}
0& 0 & k & a-k \\
0& 0 & a'-k & b-a'+k \\
k& a'-k & 0 & 0 \\
a-k & b-a'+k & 0 &0
\end{pmatrix}
\ee
and such that $a+b=a'+b'$. The latter condition does not depend on $k$, meaning that if it does not hold, then the whole expression is zero. This was obvious from the fact that $H_a(\gamma_1)H_{b}(\gamma_2)\in W_{a+b}$.
\end{proof}

\subsection{Exact formula for the variance}
Recall the notations introduced in \cref{sec:cov1jet}.
\begin{corollary}\label{cor:exactvar}
Let $f:M\to \mathbb{R}$ be as in \cref{thm:nodalchaos}. Then
\be
\begin{gathered}
\Var \tyu \lf(M)[q]\uyt 
\\
=\int_M \int_M \int_{S^{\d-1}(T_xM)}\int_{S^{\d-1}(T_yM)} 
\sum_{\substack{a,b\in \N\\ 2a+2b=q }\ }  \sum_{\substack{a',b'\in \N\\ 2a'+2b'=q }\ } 
\E\kop 
\lfs(dxdu)[a,b]\cdot \lfs(dydv)[a',b']
\pok,
\end{gathered}
\ee
and
\be
\begin{gathered}
    \E\kop 
\lfs(dxdu)[a,b]\cdot \lfs(dydv)[a',b']
\pok
\\
\begin{aligned}
=\ &\quad \qquad 
\frac{1}{s^2_n}\sum_{k=\max\{0,2a'-2b\}}^{\min\{2a,2a'\}} \frac{
{\coeff(a,b)\coeff(a',b')}
(2a)!(2b)!(2a')!(2b')!
}{k!(2a-k)!(2a'-k)!(2b-2a'+k)!} 
\\
&\times 
\tyu C_{xy}\uyt^k\tyu \frac{C'_{xy}(v)}{\|v\|_{g^f}}\uyt^{2a-k}\tyu \frac{C'_{yx}(u)}{\|u\|_{g^f}}\uyt^{2a'-k}\tyu \frac{C''_{xy}(u,v)}{\|v\|_{g^f}\|u\|_{g^f}}\uyt^{2b-2a'+k} {\|v\|_{g^f}\|u\|_{g^f}}
\ dvdudydx.
\end{aligned}
\end{gathered}
\ee
\end{corollary}
\begin{proof}
A straightforward consequence of \cref{thm:nodalchaos} and \cref{lem:diagfour}.
\end{proof}
\section{Variance bounds}

\subsection{Bounding the coefficients}
In this subsection, we show an inequality on the coefficients appearing in the covariance formula of \cref{lem:diagfour}.
Define
\be
\sigma(a,b,a',b',k):=\frac{(2a)!(2b)!(2a')!(2b')!}{k!(2a-k)!(2a'-k)!(2b-2a'+k)!}.
\ee
Then we have the following two lemmas.
\begin{lemma}\label{lem:q!}
Let $q=2a+2b=2a'+2b'\in2 \N$, then
\be 
\kappa(a,b,a',b'):=\sum_{k=\max\{0,2a'-2b\}}^{\min\{2a,2a'\}}\sigma(a,b,a',b',k)
\le q!
\ee
\end{lemma}
\begin{proof}
The inequality follows from the Vandermonde identity:
\be 
\binom{q}{\ell}=\sum_{k=0}^\ell \binom{2a}{k}\binom{2b}{\ell-k};
\ee
which yields
\bega
\kappa(a,b,a',b') &\le 
\sum_{k=0}^{2a'} {(2a')!(2b')!}\binom{2a}{k}\binom{2b}{2a'-k}\le {(2a')!(2b')!}\binom{q}{2a'}=q!
\eega
\end{proof}
\begin{lemma}\label{lem:2q}
For any $q\in 2\N$, we have that
\be
    \begin{gathered} 
\sum_{\substack{
(a,b,a',b',k)\in \N
\\
2a+2b=2a'+2b'=q
}}
|\coeff(a,b)\coeff(a',b')| q!
\le 2^q,
\end{gathered}
\ee
where $\coeff(a,b)$ is defined in \cref{eq:defC}. The equality holds only for $q=0$.
\end{lemma}
\begin{proof}
From the previous lemma we get that:
\bega 
\sum_{\substack{
(a,b,a',b',k)\in \N
\\
2a+2b=2a'+2b'=q
}}
|\coeff(a,b)\coeff(a',b')| q!
\le&
q!\tyu \sum_{\substack{
(a,b)\in \N,\ 
2a+2b=q
}}
\coeff(a,b) \uyt^2
\\
&=
2^{-q}\binom{q}{\frac{q}{2}}\tyu \sum_{\substack{
(a,b)\in \N,\ 
2a+2b=q
}}
\frac{\binom{a+b}{b}}{|2b-1|} \uyt^2
\\
&\le
2^{-q}\binom{q}{\frac{q}{2}}
\tyu \sum_{\substack{
(a,b)\in \N,\ 
2a+2b=q
}}
{\binom{a+b}{b}}{} \uyt^2
\le \frac{2^q}{\sqrt{\pi \frac{q}{2}}}\le 2^q.
\eega
In the last line we used the inequalities
\be 
|2b-1|\ge 1\quad \text{ and } \quad \binom{2a}{a}\le \frac{4^a}{\sqrt{\pi a}},
\ee
the second of which is deduced from Stirling's formula, see for instance \cite[Eq. (2)]{DUTTON1986211}.
\end{proof}

\subsection{Covariance bounds}\label{Sec:cov_bounds}
\begin{theorem}\label{thm:varbound_dxdvab}
Let $f:M\to \mathbb{R}$ be as in \cref{thm:nodalchaos}. For any $2a+2b=2a'+2b'=q\in 2\N$,
\bega
|\E\kop \lfs(dxdu)[a,b]\cdot \lfs(dydv)[a',b']\pok|
\le \frac{
\|\Lambda_xu\|\|\Lambda_y v\|
|\coeff(a,b)\coeff(a',b')|}{\n s_\n^2}\cdot q! \cdot 
\|j''_{x,y}C\|_{g^f}^q.
\eega
The inequality is meant as measure densities on $S(TM)\times S(TM)$.
\end{theorem}
\begin{proof}
We use the representation of $\lfs(dxdu)[a,b]$ given in \cref{luxab} and apply \cref{lem:diagfour}.
\be
\begin{gathered} 
|\E\kop \m L(dxdu)[a,b]\cdot \m L(dydv)[a',b']\pok|
\le 
\\
\frac{1}{\n s_\n^2}
\|\Lambda_xu\|\|\Lambda_y v\|
|\coeff(a,b)\coeff(a',b')|\sum_{k=\max\{0,2a'-2b\}}^{\min\{2a,2a'\}}\kappa(a,b,a',b',k)
\times
\\
\times \left|
C(x,y)^k\tyu\frac{\sqrt{n}C'_{x,y}(v)}{\|\Lambda_y v\|}\uyt^{2a-k}\tyu\frac{\sqrt{n}C'_{y,x}(u)}{\|\Lambda_x u\|}\uyt^{2a'-k}\tyu\frac{nC''_{x,y}(u,v)}{\|\Lambda_x u\|\|\Lambda_yv\|}\uyt^{2b-2a'+k}
\right|
dvdudydx
\\
\le 
\frac{
\|\Lambda_xu\|\|\Lambda_y v\|
|\coeff(a,b)\coeff(a',b')|}{\n s_\n^2}\cdot q! \cdot 
\|j''_{x,y}C\|_{g^f}^q
dvdudydx.
\end{gathered}
\ee
The last inequality follows from \cref{lem:q!} and \cref{eq:jtwonorm}.
\end{proof}
\begin{corollary}\label{thm:varbound_dxdv}
Let $f:M\to \mathbb{R}$ be as in \cref{thm:nodalchaos}. Then
\bega
|\E\kop \lfs(dxdu)[q]\cdot \lfs(dydv)[q]\pok|
\le \frac{
\|\Lambda_xu\|\|\Lambda_y v\|
}{\n s_\n^2}\cdot 2^q \cdot 
\|j''_{x,y}C\|_{g^f}^q,
\eega
for all $q\in 2\N$. The inequality is meant as measure densities on $S(TM)\times S(TM)$.
\end{corollary}
\begin{proof}
From \cref{thm:varbound_dxdvab} and  \cref{luxq}, we deduce that 
\be 
|\E\kop \lfs(dxdu)[q]\cdot \lfs(dydv)[q]\pok|\le \sum_{\substack{
(a,b,a',b',k)\in \N
\\
2a+2b=2a'+2b'=q
}}
|\coeff(a,b)\coeff(a',b')|q!
\frac{
\|\Lambda_xu\|\|\Lambda_y v\| 
\|j''_{x,y}C\|_{g^f}^q}{\n s_\n^2}.
\ee
Now, \cref{lem:2q} proves the thesis. 
\end{proof}
\begin{corollary}[\cref{cor:var_dx}]\label{cor:var_dx2}
Let $f:M\to \mathbb{R}$ be as in \cref{thm:nodalchaos}. Then
\bega
|\E\kop \lf(dx)[q]\cdot \lf(dy)[q]\pok|
\le \
\frac{\lambda(f,x)\lambda(f,y)}{\n}
\cdot
{2^q\|j''_{x,y}C\|_{g^f}^q}
\ dxdy.
\eega
\end{corollary}

\begin{proof}
A direct consequence of \cref{thm:varbound_dxdv} and H\"older's inequality:
\be 
\frac{1}{s_\n}\int_{S^{\d-1}(T_xM)}\|\Lambda_x v\|du
\le
\tyu\frac{1}{s_\n}\int_{S^{\d-1}(T_xM)}\|\Lambda_x u\|^2du\uyt^{\frac{1}2}=\lambda(f,x).
\ee
\end{proof}
\begin{corollary}\label{cor:var}
Let $f:M\to \mathbb{R}$ be as in \cref{thm:nodalchaos}. Then
\bega
\Var \tyu \lf(M)[q]\uyt 
\le \
{2^q}
\cdot
\int_{M\times M}{\|j''_{x,y}C\|_{g^f}^q} \frac{\lambda(f,x)\lambda(f,y)}{\n}dxdy.
\eega
\end{corollary}
\begin{proof}
A direct consequence of \cref{cor:var_dx2} and of the definition of $\lambda(f,x)$, \cref{def:pointwisefreq}. 
\end{proof}
\subsection{Reduction to the homothetic case}\label{sec:redhom}
Recall the setting of \cref{subsec:assjet}.
In this section we study the error that is commited by using the simpler expression \cref{eq:nodalchaosHomo} in the general, non-homothetic case.
\be 
\e(x,u):=
 \frac{\left\|\Lambda_x  u\right\|}{\lambda} -1, \quad \text{and}\quad \e_0(x):=\frac{\sqrt{\E\kop \phi(x)^2\pok}}{\sigma}-1
\ee
Then, up to a multiplicative factor we can take
\be 
\e(\phi)= \max_{x\in M,u\in S(T_xM)}
|\e(x,v)|+\max_{x\in M,u\in S(T_xM)}|\langle d_x\e_0,v\rangle|\frac{\sqrt{n}}{\lambda}+\max_{x\in M}\e_0(x).
\ee

\begin{proof}[Proof of \cref{thm:non-homotetic}]
The proof relies also on the variance bound established in \cref{thm:varbound_dxdvab}. 
Notice that
\be 
 \frac{\phi(x)}{\sigma}=f(x)
 (1+\e_0(x));
\ee
Let us define
\be 
\xi(x,u):=\frac{\langle \nabla_xf, u\rangle}{\|\Lambda_x u\|} \sqrt{\n}, \quad\text{and}\quad   \frac{\sqrt{n}}{\lambda}\langle \nabla_x\e_0, u\rangle=:\e'_0(x,u);
\ee
so that 
\be 
\frac{\langle \nabla_x\phi, u\rangle}{\sigma} \frac{\sqrt{\n}}{\lambda}
 =\xi(x,u)(1+\e(x,u))\tyu 1+\e_0(x))\uyt+f(x)\e'_0(x,u) ;
\ee
Consider the degree $q$ polynomial in the variables $f$, $\xi$ and $\e$, $\e_0$ and $\e_0'$
\be
\begin{gathered}
P\tyu f,\xi,\e,\e_0',\e_0\uyt:=
\\
\big( \sum_{a,b\in \N,\ a+b=\frac{q}2}\frac{\coeff(a,b)}{s_\n\sqrt{\n}} \qwe H_{2a}\tyu f\uyt H_{2b}\tyu\xi\uyt(1+\e)
-H_{2a}\tyu f(1+\e_0)\uyt H_{2b}\tyu\xi(1+\e)(1+\e_0)+f\e_0'\uyt\ewq\big) [\ell] ,
\end{gathered}
\ee
and observe that 
\be
\begin{gathered}
\frac{\hlf(M)[q]- \hlfi(M)\{q\} [\ell]}\lambda
\\
= \int_M\int_{S(T_xM)} 
 P\tyu f(x),\xi(x,u),\e(x,u),\e_0'(x,u),\e_0(x)\uyt
 du dx.
\end{gathered}
\ee
We will need the next lemma.
\begin{lemma}\label{lem:polynomi}
    Let $P(X,Y)$ be a polynomial in the $\ell^{th}$ chaos. Then, 
    \be 
|\E\kop P(f(x),\xi(x,u) )P(f(y),\xi(y,v))\pok |\le \|P\|^2 \|j''_{x,y}C\|_{g^f}^\ell,
    \ee
    where $\|P\|$ 
    %
    is a norm on the space of polynomials.
\end{lemma}
\begin{proof}
We can write the polynomial $P$ in the basis given by the Hermite polynomials:
\be
P(X,Y)=\sum_{a+b= q} p_{a,b} H_{2a}(X)H_{2b}(Y),
\ee
then, from \cref{thm:varbound_dxdvab} we deduce that
\bega
|\E\kop P(f(x),\xi(x,u) )P(f(y),\xi(y,v))\pok |
\\ 
\le 
\sum_{a+b= \ell} \sum_{a'+b'= \ell} |p_{a,b}p_{a',b'}||\E\kop H_{a}(f(x))H_{b}(\xi(x,u))H_{a}(f(y))H_{b}(\xi(y,v))\pok |
\\
\le 
 \sum_{a+b= \ell} \sum_{a'+b'= \ell} |p_{a,b}p_{a',b'}|\ell! \cdot 
\|j''_{x,y}C\|_{g^f}^\ell
\le 
\tyu \sum_{a+b= \ell}  |p_{a,b}|\uyt^2\ell! \cdot 
\|j''_{x,y}C\|_{g^f}^\ell.
\eega
 We conclude by setting
 \be 
 \sum_{a,b\in\N}  |p_{a,b}| \sqrt{(a+b)!}
 =:\|P\|.
 \ee
 $\kappa_{2\ell+1,q}=0$ because $\hlfi(M)\{q\}$ has only even chaotic components.
\end{proof}
Since $\|P\|$ is a norm on the finite dimensional space of polynomials of degree less than $\ell$, and since $P(X,Y,0,0,0)=0$, it follows that $\|P(\cdot,\cdot,\e,\e_0,\e_0')\|\le \kappa_q (|\e|+|\e_0|+|\e_0'|)$, for some constant $\kappa_q>0$. From \cref{lem:polynomi} we conclude that
\be
\begin{gathered}
\E\tyu\frac{\hlf(M)[q]- \hlfi(M)\{q\} [\ell]}\lambda\uyt^2
 \le s_{\n-1}^2 \int_{M\times M}\kappa_q^2 \e(\phi)^2\|j''_{x,y}C\|_{g^f}^\ell .
\end{gathered}
 \ee
\end{proof}

\subsection{Homothetic fields: Proof of \texorpdfstring{\cref{thm:homo24}}{Proposition \ref{thm:homo24}}} \label{subsec:proof_prop_homothetic}
In this section, we demonstrate \cref{thm:homo24}, which provides the first chaotic components of the nodal volume expansion for a homothetic field (see 
\cref{def:esphom}). We stress that these formulas can be indeed applied to any field, as argued in 
\cref{sec:introredhom}, and according to 
\cref{thm:non-homotetic}. Let us define the operator $\LL\colon \mC^\infty(M)\to \mC^\infty(M)$ such that 
\be\label{eq:defLL}
\LL f:=f+\frac{1}{\lambda^2}\Delta f,
\ee
where $\lambda=\lambda(f)$ is defined as in \cref{def:three}.
\begin{proof}[Proof of \cref{thm:homo24}]
Let $f\in \mC^\infty(M)$ be any function, then the expression of $\hlf(M)\{2\}$, defined as in \cref{def:esphom} can be easily computed as follows. We report the computation for $\sigma=1$, as the general $\sigma$ can be directly deduced from that.
\bega \label{eq:2chaos}
\hlf(M)\{2\}
=\ &
\frac{\lambda}{s_\n \sqrt{\n}} \bigg\{ -\frac{1}{2}\int_M\int_{S(T_xM)} 
 (f(x)^2-1)
dv dx 
+ \frac{1}{2}
\int_M\int_{S(T_xM)} 
(\langle \nabla_xf, v\rangle^2\frac{\n}{\lambda^2}-1)dvdx\bigg\}\\
=\ &
\frac{\lambda}{s_\n \sqrt{\n}} \bigg\{ -\frac{1}{2}\int_M
 f(x)^2
s_{\n-1}dx 
+ \frac{1}{2}
\int_M 
\| \nabla f\|^2\frac{\n s_{{\n-1}}}{\n\lambda^2}dx\bigg\}\\
=\ &
-\frac{\lambda s_{\n-1}}{2s_\n \sqrt{\n}} \bigg\{ \int_M
 f(x)^2
dx 
- 
\int_M 
\| \nabla f\|^2\frac{1}{\lambda^2}dx\bigg\}.
\eega
In the second step, we used \cref{lem:intsph}.
This proves the first identity. Now 
   applying Green's identity (recalled in \cref{lem:Riemannianformulas}) we get
    \bega \label{eq:2chaos-homothetic}
\hlf(M)\{2\}=\ &
-\frac{\lambda s_{\n-1}}{2s_\n \sqrt{\n}} \qwe \int_M
 f(x)^2
dx 
- \frac{1}{\lambda^2} \tyu \int_M - f \Delta f \,dx+ \int_{\partial M}  f \langle \nabla f, \nu\rangle \uyt \ewq\\
=\ &
-\frac{\lambda s_{\n-1}}{2s_\n \sqrt{\n}} \cdot 
\tyu\int_{M} f\LL f
- \frac{1}{\lambda^2}\int_{\partial M} f \langle\nabla f , \nu \rangle \uyt,
\eega
which completes the computation of $\hlf(M)\{2\}$.

Let us now compute the fourth order component.
\bega 
\hlf(M)\{4\}=\ &
\frac{\lambda}{s_\n \sqrt{\n} } \bigg\{ \coeff(2,0)
\int_M\int_{S(T_xM)} 
 H_4(f(x))
dv 
dx
\cr
&+
\coeff(1,1)\int_M\int_{S(T_xM)} 
 H_2(f(x))
 H_2\tyu \langle \lambda^{-1} \nabla_x f, v\rangle\sqrt{\n} \uyt dv 
dx
\cr
&+
\coeff(0,2)\int_M\int_{S(T_xM)} 
 H_4\tyu \langle \lambda^{-1} \nabla_x f, v\rangle \sqrt{\n}\uyt dvdx\bigg\}.
\eega
By definition (see \cref{eq:defC}), $\coeff (2,0)=\frac{1}{8}, \coeff (1,1)=-\frac{1}{4} $ and $\coeff (0,2)=-\frac{1}{24}$. Moreover, the last integral in $S(T_{x}M)$ can be evaluated exploiting \cref{lem:intsphH4}, hence
\bega
\hlf(M)\{4\} =\ &
\frac{\lambda}{s_\n \sqrt{\n}} \bigg\{ \frac{1}{8}
\int_M 
 (f(x)^4-6f(x)^2+3)s_{\n-1}
dx
\cr
&-\frac{1}{4}\int_M\int_{S(T_xM)} 
( f(x)^2 \langle \lambda^{-1}\nabla_xf, v\rangle ^2\n -\langle \lambda^{-1} \nabla_xf, v\rangle^2 \n -f(x)^2+1 )\, dv
dx
\cr
&-
\frac{1}{24}\int_M s_{\n-1}
 \tyu \frac{3}{\n(\n+2)}\| \lambda^{-1}\nabla f \|^4 \n^2-\frac{6}{\n}\| \lambda^{-1}\nabla f \|^2 \n+3 \uyt 
dx.
\cr
\eega
Now, applying \cref{lem:intsph} we can rewrite the quantity above as
\bega
\hlf(M)\{4\} =\ &
\frac{\lambda s_{\d-1}}{24 s_\n \sqrt{\n}} \bigg\{ \int_M 
3f^4-18f^2 +9-6f^2 \|\lambda^{-1}\nabla f\|^2 +6\|\lambda^{-1}\nabla f\|^2 +6f^2-6\\
&-
\frac{3}{\n(\n+2)}\| \lambda^{-1}\nabla f \|^4 \n^2+\frac{6}{\n}\| \lambda^{-1}\nabla f \|^2 \n-3 \, 
dx \bigg\}\cr
=\ &
\frac{s_{\d-1}}{24 s_\n \sqrt{\n}}\int_M 
3f^2(f^2-4)
- \frac{3\n}{(\n+2)}\|\lambda^{-1}\nabla f\|^4
-6{f^2\|\lambda^{-1}\nabla f\|^2}
 +
12{\|\lambda^{-1}\nabla f\|^2} \,dx\eega
that is what we wanted. Applying Green's identity and formulas in \cref{lem:Riemannianformulas} we get 
\bega \label{eq:4chaos-homothetic}
\frac{\hlf(M)\{4\}}{\lambda} &=\frac{s_{\d-1}}{24 s_\n \sqrt{\n}}\int_M 
3f^2(f^2-4)
- \frac{3\n}{(\n+2)}\|\lambda^{-1}\nabla f\|^4
+\frac{2}{\lambda^2}f \Delta f (f^2-6) dx \\ & \quad -  \frac{2}{\lambda^2} \int_{\partial M}f(f^2-6) \langle \nabla f, \nu \rangle ,\eega
which leads to the thesis recalling 
that $\frac{1}{\lambda^2}\Delta f= \LL f-f$.
\end{proof}

%

\begin{appendix}


\section{Variance of the second and fourth chaotic components}\label{appendix:homothetic}
In this section we give an expression of the variance of the firts chaotic components in terms of the covariance function of the field and their derivatives in the homothetic case. Then, let us consider $f$ a homothetic field (see \cref{sec:homotheticfields}) satisfying \cref{eq:assjet}.
We first
introduce some more notation which will be used in the next results.

\begin{definition}\label{def:LLL}
Recall the differential operator $\LL=1+\frac{\Delta}{\lambda^2}$ acting on $\mC^\infty(M)$. For $C\in \mC^{\infty}(M\times M)$, we use the notation $\LL_1C(x,y):=\LL C(\cdot, y)|_x$ and $\LL_2C(x,y):=\LL C(x,\cdot)|_y$. 
\end{definition}
Observe that the operator $\LL$ is clearly self-adjoint, being a sum of self adjoint operators, that is
\be 
\int_M \LL f(x) g(x)  \,dx = \int_M f(x) \LL g(x) \,dx ,
\ee
for any $f,g\in \m C^\infty(M)$.

Note that the random variables $f(x)$ and $\LL f (x)$ are independent.
The notation introduced in \cref{def:LLL} is useful to express the covariance function of the field $(f(x),\LL f(x))$, that is 
\be 
C_\LL(x,y):=\E\kop \begin{pmatrix}
    f(x) \\ \LL f(x)
\end{pmatrix}\begin{pmatrix}
    f(y) & \LL f(y)
\end{pmatrix}\pok=\begin{pmatrix}
    C(x,y) & \LL_2C(x,y) \\
    \LL_1C(x,y) & \LL_1\LL_2C(x,y),
\end{pmatrix}
\ee
where in particular, for $x=y$, we get
\be 
C_\LL(x,x)=\begin{pmatrix}
    1 & 0 \\
    0 & \frac{\E|\Delta f(x)|^2}{\lambda^4}-1
\end{pmatrix}.
\ee
We are now ready to compute the variance of $\m L_f[2]$.
\begin{theorem}\label{thm:homexactvar2} Let the assumptions of \cref{cor:altexprL} prevail. Then
\bega
\Var\kop \lf[2]\pok
&= \frac{\lf[0]^2 }{2\vol \d(M)^2}\cdot \bigg\{\int_M\int_{M}
\LL_2C(x,y)\LL_1C(x,y) dx dy\cr 
&\quad + \frac{1 }{2\lambda^4}\cdot  \int_{\partial M} \int_{\partial M} \tyu C_{x,y}'(\nu_x)C_{x,y}'(\nu_y)+ C(x,y) C_{x,y}''(\nu_x,\nu_y) \uyt dx dy \cr
&\quad - \frac{1 }{\lambda^2}\cdot \int_{M}\int_{\partial M} \tyu  C_{x,y}'(\nu_y) \LL_1C(x,y) +C(x,y) \LL_1C_{x,y}'(\nu_y) \uyt  dydx\bigg\}.
\eega
\end{theorem}
\begin{proof}
From  \cref{eq:2chaos-homothetic} it follows that
\bega \label{eq:var2}
\Var\kop \lf[2]\pok &=
\frac{\lf[0]^2 }{4 \vol \d(M)^2} \cdot \bigg\{ \Var\tyu
\int_{M} f\LL f\uyt+ \frac{1}{\lambda^2 \n}\Var\tyu \int_{\partial M} f \langle \lambda^{-1}\sqrt{n}\nabla f , \nu \rangle \uyt\cr &
\quad - \frac{2 }{\lambda\sqrt{\n}} \cdot \Cov \tyu \int_{M} f\LL f,\int_{\partial M} f \langle \lambda^{-1}\sqrt{\n}\nabla f , \nu \rangle \uyt\bigg\}.\eega
We can apply \cref{lem:diagfour}, since $f(x)$ and $\LL f(x)$ are independent. Then for the variance of the first term we get
\bega 
\Var\tyu \int_M f\LL f\uyt
&=
\int_M\int_{M} \E\kop f(x)\cdot\LL f(x)\cdot f(y)\cdot\LL f(y)\pok dxdy
\cr
&=
\int_M\int_{M} \tyu 
\LL_2C(x,y)\LL_1C(x,y)+C(x,y) \LL_1\LL_2C(x,y)
\uyt dxdy
\cr
&=
2\int_M\int_{M}
\LL_2C(x,y)\LL_1C(x,y) dxdy.
\eega
Similarly, for the variance of the second term and the covariance term in \cref{eq:var2} we have
\begin{align*}
\Var\tyu \int_{\partial M} f \langle \lambda^{-1}\sqrt{\n} \nabla f , \nu \rangle \uyt &= 
\int_{\partial M}\int_{\partial M} \mathbb{E}[ f (x)\langle \lambda^{-1}\sqrt{\n} \nabla_x f, \nu_x \rangle f(y) \langle\lambda^{-1}\sqrt{\n}\nabla_yf, \nu_y \rangle ] dxdy
\cr
&=\frac{\n}{\lambda^2}\int_{\partial M}\int_{\partial M} \tyu C'_{x,y}(\nu_x) C'_{x,y}(\nu_y) +C(x,y)C''_{x,y}(\nu_x,\nu_y) \uyt dxdy;
\end{align*}
\begin{equation*}
\begin{aligned}
& \Cov \tyu \int_{M} f\LL f , \int_{\partial M} f \langle \lambda^{-1}\sqrt{\n}\nabla f , \nu \rangle \uyt 
=
\int_M \int_{\partial M} \mathbb{E} [ f(x)\LL f(x)f(y)\langle \lambda^{-1}\sqrt{\n}\nabla_y f , \nu_y \rangle ] dydx\cr
& =\frac{\sqrt{\n}}{\lambda}\int_{M}\int_{\partial M} \tyu C_{x,y}'(\nu_y) \LL_1C(x,y) +C(x,y)\LL_1C'_{x,y}(\nu_y) \uyt \,dydx .
\end{aligned}
\end{equation*}
Putting together these results we complete the proof.
\end{proof}
Let us now compute the variance of $\lf[4]$. To do that we want to apply \cref{lem:diagfour} as we have done for $\lf[2]$. To this aim it is useful to express $\lf[4]$ in terms of Hermite polynomials.
\begin{remark}
Starting from \cref{eq:4chaos-homothetic},
we get the following alternative formula
\bega
\label{eq:4chaos_hom_L_Hermite}\frac{\lf[4]}{\lambda}
=&
\frac{s_{\n-1}}{24s_\n \sqrt{\n}} \bigg\{ 
\int_M\Bigg[ 
H_4(f(x)) - \frac{1}{s_{\n-1}}
\int_{S(T_xM)} 
 H_4\tyu \langle \lambda^{-1}\sqrt{\n}\nabla_x f, v\rangle \uyt dv  
\cr & \quad  +2H_3(f(x))\LL f (x)
\Bigg]
dx\cr 
& \quad -  \frac{2}{\lambda\sqrt{\n}}\int_{\partial M}\tyu H_3(f(x))\langle \lambda^{-1}\sqrt{\n}\nabla_x f, \nu_x
\rangle \uyt dx \bigg\}.
\eega
\end{remark}
\begin{theorem}\label{thm:homexactvar4} Let the assumptions of \cref{cor:altexprL} prevail. Then
\bega
& \tyu\frac{s_{\n-1}\lambda}{24 s_\n \sqrt{\n}}\uyt^{-2}\Var\tyu \lf[4]\uyt 
=\cr&
\int_M\int_M \bigg\{ 4!C(x,y)^4
+
\frac{4!}{3} \beta(\d,4)^2\qwe 2
\tr\tyu\tyu { \lambda^{-4}\n^2 C''_{x,y}}^TC''_{x,y}\uyt^2\uyt  +\tr\tyu {\lambda^{-4}\n^2C''_{x,y}}^TC''_{x,y}\uyt ^2 
 \ewq
 \cr
&\quad -
2\cdot 4! \lambda^{-4}\n^2\cdot 
 \|C'_{x,y}\|^4
 \beta(\d,4)
\cr
&\quad+
4 \times 36\{ \frac{1}{2!} C(x,y)^2\LL_2C(x,y)\LL_1C(x,y)+\frac{1}{3!} C(x,y)^3\LL_1\LL_2C(x,y) \} 
\cr
&\quad+
4 \times  4! C(x,y)^3 \LL_2C(x,y)
\cr
&\quad-
4\times 4! \beta(d,4) \lambda^{-4} \n^2 \|C_{x,y}'\|^2\langle C_{x,y}',\LL_1 C_{x,y}'\rangle\bigg\} \,dx\,dy\cr
&
+\frac{12}{\lambda^4 }  \int_{\partial M} \int_{\partial M} 3! C(x,y)^2 C'_{x,y}(\nu_y)C'_{x,y}(\nu_x)+ 4 C(x,y)^3 C''_{x,y}(\nu_x,\nu_y) \, dx \,dy \cr
&
-\frac{48}{\lambda^2} \int_M\int_{\partial M} \bigg(2 C(x,y)^3 C'_{x,y} (\nu_y) + 3
C(x,y)^2 C'_{x,y}(\nu_y)\LL_1C(x,y)+  C(x,y)^3\LL_1C'_{x,y}(\nu_y) \, \cr
&
 \quad -\frac{2\lambda^{-4}\n^2}{s_{\n-1}}\int_{S(T_xM)} C_{x,y}'(v)^3 C''_{x,y}(v,\nu_y) dv \bigg)dy\,dx,
 \eega 
where $\beta(\n,4)=\int_{S^{\d-1}} |v_1|^4 \mathcal{H}^{\d-1}(dv)= \frac{3\pi^{\frac{\n}{2}}}{2\Gamma(\frac{\n+4}{2})}$, as computed in \cref{lem:beta}.
\end{theorem}
\begin{proof}
In view of \cref{eq:4chaos_hom_L_Hermite} we can write  
 \bega \label{var4}& \tyu\frac{s_{\n-1}\lambda}{24 s_\n \sqrt{\n}}\uyt^{-2}\Var\tyu \lf[4]\uyt \cr
 &= \Var \bigg(
\int_M\Bigg[ 
H_4(f(x))+2H_3(f(x))\LL f (x)- \frac{1}{s_{\n-1}}
\int_{S(T_xM)} 
 H_4\tyu \langle \lambda^{-1}\sqrt{\n}\nabla_x f, v\rangle \uyt dv
\Bigg]
dx \bigg) \cr 
& \quad +  \frac{4}{\lambda^2 \n} \Var \bigg(\int_{\partial M} H_3(f(x))\langle \lambda^{-1} \sqrt{\n}\nabla_x f, \nu_x \rangle \,dx \bigg) \cr
& \quad -\frac{4}{\lambda\sqrt{\n}} \Cov \bigg( \int_M\Bigg[ 
H_4(f(x))+2H_3(f(x))\LL f (x) 
-
\frac{1}{s_{\n-1}}
\int_{S(T_xM)} 
 H_4\tyu \langle \lambda^{-1}\sqrt{\n}\nabla_x f, v\rangle \uyt dv
\Bigg]
dx , \cr
&\quad \quad 
 \quad \int_{\partial M} H_3(f(x))\langle \lambda^{-1}\sqrt{\n}\nabla_x f, \nu_x
\rangle \,dx \bigg).
 \eega
To determine the first variance, we need to evaluate the following covariance functions, which can be calculated exploiting \cref{lem:diagfour}
\begin{equation}
\begin{aligned}
\label{covfns}
& \E\qwe 
 H_4(f(x))
H_4(f(y))
 \ewq =4!C(x,y)^4,\cr
 &\E\qwe 
 H_4(f(x)) H_4\tyu \langle \lambda^{-1}\sqrt{\n}\nabla_yf, v\rangle \uyt
 \ewq =4!(\lambda^{-1}\sqrt{n}C'_{x,y}(v))^4,\cr
 &\E\qwe  H_4\tyu \langle \lambda^{-1}\sqrt{\n}\nabla_xf, u\rangle \uyt
 H_4\tyu \langle \lambda^{-1}\sqrt{\n}\nabla_yf, v\rangle \uyt
 \ewq =4!(\lambda^{-2}\n C''_{x,y}(u,v))^4,
\cr
&\E\qwe 
 H_3(f(x))\LL f (x)
H_3(f(y))\LL f (y)
 \ewq 
 = 36\{ \frac{1}{2!} C(x,y)^2\LL_2C(x,y)\LL_1C(x,y)\\ &\quad +\frac{1}{3!} C(x,y)^3\LL_1\LL_2C(x,y) \} \text{,}
\cr &
\E\qwe 
 H_4(f(x))
 H_3(f(y))\LL f (y)
 \ewq= 
 4!3! \frac{ C(x,y)^3 \LL_2C(x,y) }{3!}\text{,}\cr
 &\E\qwe 
H_3(f(x))\LL  f (x)
 H_4\tyu \langle \lambda^{-1}\sqrt{\n}\nabla_yf, v\rangle \uyt
 \ewq \cr &\quad = 4! \mathbb{E}[f(x)\langle \lambda^{-1}\sqrt{\n} \nabla_y f,v \rangle]^3 \mathbb{E}[\LL f(x) \langle\lambda^{-1}\sqrt{\n} \nabla_y f,v \rangle] \cr & \quad =4! (\lambda^{-1}\sqrt{\n}C_{x,y}'(v))^3 \LL_1 \lambda^{-1}\sqrt{\n}C_{x,y}'(v).
 \end{aligned}
\end{equation}
Note that the missing terms $\E\qwe H_4(f(y))H_4(\langle\lambda^{-1}\sqrt{\n} \nabla_x f,u \rangle)\ewq$, 
$\E\qwe H_4(f(y))H_3(f(x))\LL f(x)\ewq$ and 
$\E\qwe  H_3(f(y)) \LL f(y) H_4\langle\lambda^{-1}\sqrt{\n} \nabla_x f,u \rangle\ewq$, when integrated, give the same contribution of the terms $\E\qwe  H_4(f(x))H_4(\langle\lambda^{-1}\sqrt{\n} \nabla_y f,u \rangle)\ewq$,  $\E\qwe H_4(f(x))H_3(f(y))\LL f(y)\ewq$ and 

$\E\qwe 
 H_3(f(x))\LL  f (x)
 H_4\tyu \langle \lambda^{-1}\sqrt{\n}\nabla_yf, v\rangle \uyt
 \ewq $, respectively. This can be seen by performing the change of variable $(x,y,u,v)\mapsto (y,x,v,u)$. Hence,
exploiting the expressions in \cref{covfns} we can write the first variance appearing in the r.h.s of \cref{var4} as
\bega
&
\int_M\int_M 4!C(x,y)^4
+
\int_{S(T_xM)}\int_{S(T_yM)}4! \qwe \lambda^{-2}\n 
C''_{x,y}(u,v)
 \ewq^4 dudv \\
 & \quad 
-
2\int_{S(T_yM)}4!\qwe 
\lambda^{-1}\sqrt{\n}  C'_{x,y}(v)
 \ewq^4 dv
\cr
&+
4 \times 36\{ \frac{1}{2!} C(x,y)^2\LL_2C(x,y)\LL_1C(x,y)+\frac{1}{3!} C(x,y)^3\LL_1\LL_2C(x,y) \} 
\cr
&+
4 \times  4! C(x,y)^3 \LL_2C(x,y)
\cr
&-
4\times 4! \int_{S(T_yM)} \lambda^{-4}\n^2 C_{x,y}'(v)^3 \LL_1 C_{x,y}'(v) dv \quad 
dx dy.
\eega 
Now, \cref{lem:termCsecondo}, \cref{lem:intsph4} and \cref{lem:term31} allow to compute the integrals in $S(T_xM)$ and $S(T_yM)$, obtaing
\bega
&\int_M\int_M 4!C(x,y)^4
+
4! \beta(\d,4) \bigg[
\tr \tyu\tyu \lambda^{-4}n^2{C''_{x,y}}^TC''_{x,y}\uyt^2\uyt \frac{2}{3}\beta(\d,4) 
\\ &\quad\quad  
+\tr \tyu \lambda^{-4}n^2{C''_{x,y}}^TC''_{x,y}\uyt^2 \frac{\beta(\d,4)}{3}
 \bigg] \cr
&-
2\cdot 4!\cdot 
\|\lambda^{-1}\sqrt{\n}C'_{x,y}\|^4
 \beta(\d,4)
\cr
&+
4 \times 36\{ \frac{1}{2!} C(x,y)^2\LL_2C(x,y)\LL_1C(x,y)+\frac{1}{3!} C(x,y)^3\LL_1\LL_2C(x,y) \} 
\cr
&+
4 \times  4! C(x,y)^3 \LL_2C(x,y)
\cr
&-
4\times 4! \beta(\d,4) \lambda^{-4}\n^2\|C_{x,y}'\|^2\langle C_{x,y}',\LL_1 C_{x,y}'\rangle \, dx \,dy.
\eega
By employing the same steps, we show that the second variance appearing on the r.h.s. of  \cref{var4} can be expressed by
\begin{align*}
&\Var \bigg(\int_{\partial M}\tyu H_3(f(x))\langle \lambda^{-1} \sqrt{\n}\nabla_x f, \nu_x
\rangle \uyt dx\bigg) \cr
&=
\int_{\partial M} \int_{\partial M} 3!3!  \bigg(\frac{1}{2!} \mathbb{E}[f(x)f(y)]^2 \mathbb{E}[f(x)\langle \lambda^{-1} \sqrt{\n} \nabla_y f, \nu_y\rangle]\mathbb{E}[\langle \lambda^{-1} \sqrt{\n} \nabla_x f, \nu_x\rangle f(y)]\cr &\quad 
+ \frac{1}{3!} \mathbb{E}[f(x)f(y)]^3 \mathbb{E}[\langle \lambda^{-1} \sqrt{\n}\nabla_x f, \nu_x\rangle \langle \lambda^{-1} \sqrt{\n} \nabla_y f, \nu_y\rangle ]\bigg) dx dy \cr
&=
\int_{\partial M} \int_{\partial M} 3!3! \tyu \frac{1}{2!} C(x,y)^2 \lambda^{-2}\n C'_{x,y}(\nu_y)C'_{x,y}(\nu_x)+ \frac{1}{3!} C(x,y)^3 \lambda^{-2}\n C''_{x,y}(\nu_x,\nu_y)\uyt dx dy
\end{align*}
and similarly the covariance term in \cref{var4} is given by
\begin{align*}
&\int_M\int_{\partial M} \bigg(4! C(x,y)^3 \lambda^{-1} \sqrt{\n} C'_{x,y} (\nu_y) \\& \quad + 3!3!
C(x,y)^2 \lambda^{-1} \sqrt{\n} C'_{x,y}(\nu_y)L_1C(x,y)  + 2 \cdot 3!C(x,y)^3 \lambda^{-1} \sqrt{\n} L_1C'_{x,y}(\nu_y)\\&
 \quad -\frac{4!}{s_{\n-1}}\int_{S(T_xM)} (\lambda^{-1}\sqrt{\n}C_{x,y}'(v))^3 \lambda^{-2} \n C''_{x,y}(v,\nu_y) dv \bigg)dydx.
\end{align*}
Putting all the results together concludes the proof.
\end{proof}
\section{Auxiliary Lemmas}\label{appendix:auxiliary}
In this appendix, we provide a compilation of technical lemmas that are used in the proofs presented in the previous sections. These lemmas focus on the evaluation of integrals over the hypersphere of real vectors, of norms of these vectors, and of powers of these norms. This foundational material was essential for analyzing the chaotic components in \cref{subsec:proof_prop_homothetic} and the variances presented in \cref{appendix:homothetic}. Finally, we recall (see \cref{lem:Riemannianformulas}) some known results on integration over a manifold. We define 
$$\beta(\n,q):=\int_{S^{\n-1}} |v_1|^q dv.$$
\begin{lemma}\label{lem:intsph}
Let $x\in \R^\d$, $\d\ge 1$ then 
\be 
\int_{S^{\d-1}} x^Ty \ dy=0; \quad \int_{S^{\d-1}} \tyu x^Ty \uyt^2 \ dy=|x|^2\beta(\d,2)=\frac{s_{\d-1}}{\d}|x|^2.
\ee 
\be 
\int_{S^{\d-1}}\int_{S^{\d-1}} \tyu x^TAy \uyt^2 \ dydx=\frac{s_{\d-1}}{\d}\int_{S^{\d-1}} |Ax|^2 \ dx=\frac{s_{\d-1}^2}{\d^2}\tr(A^TA).
\ee 
\end{lemma}
\begin{proof}
The first two formulas are a special case of \cref{lem:intsph4}, proved below in more generality. For the third, consider the singular value decomposition (SVD) of $A$. This procedure ensures that there exist two orthogonal matrices $U,V\in O(\d)$ such that $A=U\Sigma V$, where $\Sigma$ is a diagonal matrix whose entries $\sigma_1,\dots,\sigma_\d\ge 0$ are unique (up to reordering) and called the singular values of $A$. Then, performing the change of variables $v=Vx$ and using the fact that $|U\Sigma v|=|\Sigma v|$, we get that
\bega 
\int_{S^{\d-1}} |Ax|^2 \ dx
&=
\int_{S^{\d-1}} |U\Sigma Vx|^2 \ dx
=
\int_{S^{\d-1}} \sigma_1^2v_1^2+\dots +\sigma_\d^2 v_\d^2 \ dv
\\
&=
\tyu \sigma_1^2+\dots +\sigma_\d^2\uyt \beta(\d,2)=\tr(A^TA)\frac{s_{\d-1}}{\d}.
\eega
\end{proof}

Recall that 
$B(z,w)=\frac{\Gamma(z)\Gamma(w)}{\Gamma(z+w)}$ is the Beta function.
\begin{lemma}\label{lem:intsph4}
Let $x\in \R^\d$ and $q$ even, then 
\be 
\int_{S^{\d-1}} (x^Ty)^q \ dy=\|x\|^q\beta(\d,q)=\|x\|^q2\pi^{\frac{\d-1}{2}}\frac{\Gamma\tyu \frac{q}{2}+\frac{1}{2}\uyt}{\Gamma\tyu \frac{q}{2}+\frac{\d}{2}\uyt}.
\ee
In particular, with $q$ even, the function $\beta$ has the following properties:
\be 
\beta(\d,4)=\frac{3}{\d(\d+2)}s_{\d-1}
\ee
\be 
B\tyu \frac{\d-1}{2},\frac{q+1}{2}\uyt=\int_0^\pi (\cos\theta)^q(\sin\theta)^{\d-2}d\theta=s_{\d-2}^{-1}\beta(\d,q);
\ee
\be 
\beta(\d,q+2)=\frac{q+1}{q+\d}\beta(\d,q);\quad \beta(\d+2,q)=\frac{2\pi}{q+\d}\beta(\d,q)
\ee
\end{lemma}
\begin{proof}
Observe that for any $x\in \R^\d$ there exists an orthogonal matrix $R\in O(\d)$ such that $Rx=\|x\|e_1$, so that after the change of variables $y'=R^Ty$, we get
\be 
\int_{S^{\d-1}} (x^Ty)^q \ dy=\|x\|^q\int_{S^{\d-1}} y_1^q \ dy=\|x\|^q\beta(\d,q).
\ee 
Now, the lemma is concluded since the constants $\beta(\d,q)$ is computed in \cref{lem:beta}.
\end{proof}
\begin{lemma}\label{lem:termCsecondo}
Let $x\in \R^\d$, $\d\ge 1$ then 
\bega 
\int_{S^{\d-1}} \| Ax \|^4 \ dx
&= \tr\tyu (A^TA)^2\uyt \beta(\d,4)+ \tyu \tr\tyu A^TA\uyt^2-\tr\tyu (A^TA)^2\uyt\uyt\frac{\beta(\n,4)}{3}\\
&=\frac{2}{3}\tr\tyu (A^TA)^2\uyt \beta(\d,4)+ \tr\tyu A^TA\uyt^2\frac{\beta(\n,4)}{3}.
\eega
\end{lemma}
\begin{proof}
Consider the singular value decomposition (SVD) of $A$. This procedure ensures that there exist two orthogonal matrices $U,V\in O(\d)$ such that $A=U\Sigma V$, where $\Sigma$ is a diagonal matrix whose entries $\sigma_1,\dots,\sigma_\d\ge 0$ are unique (up to reordering) and called the singular values of $A$. Then, performing the change of variables $v=Vx$ and using the fact that $|U\Sigma v|=|\Sigma v|$, we get that
\bega 
\int_{S^{\d-1}} |Ax|^4 \ dx
&=
\int_{S^{\d-1}} |U\Sigma Vx|^4 \ dx
=
\int_{S^{\d-1}} \tyu \sigma_1^2v_1^2+\dots +\sigma_\d^2 v_\d^2 \uyt^2 \ dv
\\
&=
\tyu \sigma_1^4+\dots +\sigma_\d^4\uyt \beta(\d,4)+ \tyu 2\sum_{i=1}^\d \sigma_i^2\sigma_j^2\uyt \tyu\int_{S^{\d-1}}v_1^2v_2^2 dv\uyt
.
\eega
Observe that:
\be 
\tyu \sigma_1^4+\dots +\sigma_\d^4\uyt=\tr\tyu (A^TA)^2\uyt;\quad 2\sum_{1\le i\neq j\le \d} \sigma_i^2\sigma_j^2=\tr\tyu A^TA\uyt^2-\tr\tyu (A^TA)^2\uyt,
\ee
so that we conclude by computing $\a(\d,4):=\int_{S^{\d-1}} v_1^2v_2^2 dv$. In the next integral we perform the change of variables $v=(\cos \theta,\sin \theta u)$, with $u\in S^{\n-2}$, hence $dv=\sin^{\d-2}\theta du d\theta $
\bega
\alpha(\n,4)&=\int_{S^{\n-1}}v_1^2v_2^2 dv
\\
&=\int_{0}^\pi \cos^2\theta \tyu\int_{S^{\n-2}} \sin^2\theta u_1^2 du\uyt \sin^{\n-2} \theta d\theta
\\
&=\beta(\n-1,2)2\int_{0}^{\frac{\pi}2} \cos^2\theta \sin^\n\theta  d\theta.
\eega
Using the properties recalled in \cref{{lem:intsph4}} we conclude that
\begin{align*}
\a(\n,4)&=\frac{\Gamma\tyu \frac{3}{2}\uyt\Gamma\tyu \frac{\n+1}{2}\uyt}{\Gamma\tyu \frac{\n+4}{2}\uyt}\beta(\d-1,2)
=\frac{s_{\n-1}}{\n(\n+2)}=\frac{\beta(\n,4)}{3}.
\end{align*}
\end{proof}
Recall that 
$
H_4(x)=x^4-6x^2+3.
$
\begin{lemma}\label{lem:intsphH4} 
Let $x\in \R^\d$, then
\begin{align*} 
\int_{S^{\d-1}} H_4\tyu x^Ty\uyt \ dy= \tyu \frac{3}{\d(\d+2)}|x|^4-\frac{6}{\d}|x|^2+3
\uyt  s_{\d-1}.
\end{align*}
\end{lemma}

\begin{lemma}\label{lem:term31}
Let $X,Y\in \R^\d$, then
$$
\int_{S^{\d-1}}(X^Tv)^3Y^Tv dv= |X|^2X^TY\beta(\d,4).
$$
\end{lemma}
\begin{proof}
By rotational invariance we can assume that $X=|X|e_1$ and $Y=|X|^{-1}X^TY e_1+c e_2$, so that 
\begin{align*}
\int_{S^{\d-1}}(X^Tv)^3Y^Tv dv&=
|X|^3\int_{S^{\d-1}}v_1^3\tyu |X|^{-1}X^TYv_1 +cv_2\uyt dv 
=
|X|^2X^TY\int_{S^{\d-1}}v_1^4 dv. 
\end{align*}
\end{proof}
\begin{lemma}[Riemannian formulas]\label{lem:Riemannianformulas}
    \bega 
\int_M \langle \nabla f, X\rangle&=-\int_M f \div{X}+\int_{\de M} f\langle X,\nu \rangle
\\
\div{fX}&=f\div X+\langle \nabla f, X\rangle
\\
\Delta f &=\div{\nabla f}
\\
\int_M f^n\|\nabla f\|^2 &=-\frac{1}{n+1}\int_M f^{n+1}\Delta f+\frac{1}{n+1}\int_{\de M} f^{n+1}\langle \nabla f,\nu\rangle
\\
\int_M {f^{n+2}}- \frac{(n+1)}{n}f^n\|\nabla f\|^2 &=\int_M {f^{n+1}}\LL f-\frac{1}{n}\int_{\de M}f^{n+1}\langle \nabla f,\nu\rangle ,
    \eega
    where $\LL$ is the differential operator $\LL=1+\frac{\Delta}{\lambda^2}$, defined in \cref{eq:defLL}.
\end{lemma}
\begin{proof}
    The first three are standard formulas. The fourth is deduced from
\begin{align*} 
\int_M f^n\|\nabla f\|^2=\int_M \brkt{f^n\nabla f}{\nabla f}=-\int_M\div{f^n\nabla f}f=-\int_Mf^{n+1}\Delta f+nf^n\brkt{\nabla f}{\nabla f}.
\end{align*}
\end{proof}
\section{The constants}
\subsection{Computing the constant \texorpdfstring{$\beta (\d,q)$}{beta}}
We want to compute $\beta (\d,q):=\int_{S^{\d-1}} |v_1|^q \,dv$ which appears in the proof of \cref{lem:circleavg1} and along \cref{appendix:auxiliary}.
\begin{lemma}\label{lem:beta}
For any $\d,q\in \N$, then 
we have
\be 
\beta(\d,q)= 2\pi^{\frac{\d-1}{2}}\frac{\Gamma\tyu \frac{q+1}{2}\uyt}{\Gamma\tyu \frac{q+\d}{2}\uyt}, \quad \beta(\n,1)=\frac{s_\n}{\pi}.
\ee
\end{lemma}
\begin{proof}
We use the formula \cite[Eq. 631, Section 5.5]{CRCformulae}, involving the Beta function $B(z,w)=\frac{\Gamma(z)\Gamma(w)}{\Gamma(z+w)}$. Moreover, we recall that $s_{\d-2}=\vol {n-2}(S^{\d-2})=\frac{2 \pi^{\frac{\d-1}2}}{\Gamma(\frac{\d-1}2)}$. 
Let us write $v=\cos\theta e_1+\sin\theta u$, for some $u\in \{0\}\times S^{\d-2}$, then 
\bega 
\beta(\d,q) &=\int_{S^{\d-1}} |v_1|^q \,dv=
s_{\d-2}\int_0^\pi |\cos\theta|^q \sin^{\d-2}\theta d\theta
\\
&=s_{\d-2}2\int_0^{\frac{\pi}2} \cos^q\theta \sin^{\d-2}\theta d\theta
=s_{\d-2} B\tyu \frac{q+1}{2},\frac {\d-1}2\uyt
\\
&= 2\pi^{\frac{\d-1}{2}}\frac{\Gamma\tyu \frac{q}{2}+\frac{1}{2}\uyt}{\Gamma\tyu \frac{q}{2}+\frac{\d}{2}\uyt}.
\eega
\end{proof}
\subsection{Computing the constant \texorpdfstring{$c_{\chi}(2b)$}{c}}
We recall that  $c_{\chi}(2b)$ is the constant of the Wiener chaos expansion of a chi random variable and it appears in the proof of \cref{thm:genchicaos}, where we have, for $\gamma \sim N(0,1)$, $$
|\gamma|=\sum_{b\in \N} c_\chi(2b)H_{2b}(\gamma).$$  The following lemma holds.
\begin{lemma}\label{lem:chi}
    $$c_{\chi}(2b)=\frac{(-1)^{b-1}}{2^{b-1}\sqrt{2\pi}(b-1)b!} $$
\end{lemma}
\begin{proof}
   By definition of Wiener chaos coefficients we have that $$c_\chi(2b)=\frac{\E[|\gamma|H_{2b}(\gamma)]}{(2b)!}.$$ 
We can write $\E\qwe |\gamma +t|\ewq$ in two ways: using \cref{eq:H}
   and the Cameron-Martin Theorem (or just by writing the integral), or using the formula for the expectation of a non-centered chi random variable with $\d$ degrees of freedom, see \cite[Equation (A.3)]{Cammarota2017NodalAD}.
\bega \label{eq:E1} 
\sum_{n\in \N}\frac{t^n}{n!} \E\qwe |\gamma| H_n( \gamma)\ewq 
&=
\E\qwe |\gamma|e^{\gamma tv -\frac{t^2}{2}}\ewq=\E\qwe |\gamma +t|\ewq
\\
&= \int_{-\infty}^{\infty} |\gamma+t|\frac{1}{\sqrt{2\pi}} e^{-\gamma^2/2} d\gamma
\\
&=t (2\Phi(t)-1) +\frac{2}{\sqrt{2\pi}}
e^{-\frac{t^2}{2}} ,
\eega
where $\Phi (\cdot)$ is the cumulative distribution function of the standard gaussian random variable.
Now we use the Taylor expansion of $\Phi(t)$ and $e^{-t^2/2}$ to get
\bega
t (2\Phi(t)-1) +\frac{2}{\sqrt{2\pi}}
e^{-\frac{t^2}{2}}=\frac{2}{\sqrt{2\pi}} \sum_{n=0}^{\infty} \frac{(-1)^n t^{2n+2}}{2^nn!(2n+1)}+\frac{2}{\sqrt{2\pi}} \sum_{n=0}^{\infty} \frac{(-1)^n t^{2n}}{2^n n!}.
\eega
We rescale the first series setting $n=m-1$ 
and then putting together the two sums 
we obtain
\bega
&\frac{2}{\sqrt{2\pi}} \sum_{n=0}^{\infty} t^{2n}\frac{(-1)^{n-1} 2n+(-1)^n(2n-1)}{2^{n}n!(2n-1)}=\frac{2}{\sqrt{2\pi}} \sum_{n=0}^{\infty} \frac{t^{2n}}{(2n)!} \frac{(-1)^{n-1} (2n)!}{2^{n}n!(2n-1)}  .
\eega
Therefore if $q$ is even, we have
\be\label{eq:E2} 
\E\qwe |\gamma| H_q(\gamma)\ewq=\frac{2}{\sqrt{2\pi}}  \frac{(-1)^{\frac{q}{2}-1} q!}{2^{\frac{q}{2}}(\frac{q}{2})!(q-1)} 
\ee
and this concludes the proof.
\end{proof}
\section{Comparison with Laguerre polynomials}\label{sec:laguerre}
\begin{definition}
Let $S_{\d,q}\colon \R^{\d}\to \R$, be the function (it is a polynomial of degree $q$)
\be 
S_{\d,q}(x):=\int_{S^{\d-1}}H_q\tyu \langle x, v\rangle\uyt dv.
\ee
\end{definition}

The argument used in the proof of \cref{thm:genchicaos} 
shows that in fact, any $L^2$ random variable of the form $F(\|\xi\|)$ has a chaotic decomposition of the form
\be 
F(\|\xi\|) =\sum_{q\in \N } c_q S_{\d,q}(\xi).
\ee
In particular, since for all even $q$, the Laguerre polynomial $L_{\frac q2}^{(\frac \d2 -1)}\tyu \frac{\|x\|^2}{2}\uyt$ is in the $q^{th}$ chaos space $W_q[\xi]$ (see \cite[Lemma 4.3]{marinucci2023laguerre}), there exist constants $c_{\d,q}\in \R$ such that
\be\label{eq:LaS}
L_{\frac q2}^{(\frac \d2 -1)}\tyu \frac{\|\xi\|^2}{2}\uyt=c_{\d,q} \cdot S_{\d,q}(\xi), \quad \forall \xi \in \R^{\d}.
\ee
The constant $c_{\d,q}$ can be deduced by computing the variance of $L_{\frac{q}{2}}^{(\frac \d2 -1)}\tyu \frac{\|\xi\|^2}{2}\uyt$ and it is evaluated in \cref{lem:cd} here below.

\begin{lemma}\label{lem:cd}
    \be \label{eq:cdq}
c_{\d,q}= 
\frac{(\frac{q}{2}-1+\frac{\d}{2})\cdot \dots \cdot (\frac{\d}{2})}{\frac{q}{2}!(q-1)!!}(-1)^{\frac q2}\frac{1}{s_{\d-1}}
    \ee
\end{lemma}


Only for $q\in 2\N$. Recall that $s_{\d-1}=\vol {\n-1}(S^{\d-1})=\frac{2 \pi^{\frac \d2}}{\Gamma(\frac \d2)}$. Computing  the variance of a Laguerre polynomial or evaluating at $\xi=0$ the identity in \cref{eq:LaS}, we get
\bega 
\frac{\Gamma(\frac{q}{2}+\frac{\d}{2})}{\frac{q}{2}!\Gamma(\frac{\d}{2})}=L_{\frac{q}{2}}^{(\frac \d2 -1)}\tyu 0\uyt
&=c_{\d,q} \cdot S_{\d,q}(0)=c_{\d,q}s_{\d-1}(-1)^{\frac q2}(q-1)!!
\eega
which leads to \cref{eq:cdq}.

\end{appendix}
\bibliographystyle{abbrv}
\bibliography{Shermite.bib}

\end{document}